\documentclass[12pt,reqno]{amsart}
\usepackage{amsmath, amssymb, amsthm, amsfonts} 
\usepackage{url}
\usepackage[breaklinks]{hyperref}
\usepackage{tikz}
\usepackage{mathtools}
\usepackage{tikz-cd}
\usepackage{float}
\usetikzlibrary{positioning}
\usetikzlibrary{decorations.markings}
\usepackage{array}
\newcommand{\Arg}{\operatorname{Arg}}

\renewcommand{\Re}{\operatorname{Re}}

\renewcommand{\(}{\left\(}
\renewcommand{\)}{\right\)}
\renewcommand{\[}{\left\[}
\renewcommand{\]}{\right\]}

\numberwithin{equation}{section}
 \theoremstyle{plain}
\newtheorem{theorem}{Theorem}[section]
\newtheorem{lemma}[theorem]{Lemma}
\newtheorem{remark}[]{Remark}

\newtheorem{corollary}[theorem]{Corollary}

   \makeatletter
\def\proof{\@ifnextchar[{\@oproof}{\@nproof}}
\def\@oproof[#1][#2]{\trivlist\item[\hskip\labelsep\textit{#2 Proof of\
#1.}~]\ignorespaces}
\def\@nproof{\trivlist\item[\hskip\labelsep\textit{Proof.}~]\ignorespaces}

\makeatother

\begin{document}
\title[Zeros of Dedekind zeta function on $\Re(s)=1/2$]{Number Field Analogue of Jacobi Theta Relation And Zeros of Dedekind zeta function on $\Re(s)=1/2$}



\author{Diksha Rani Bansal}
\address{Diksha Rani\\ Department of Mathematics \\
Indian Institute of Technology Indore \\
Indore, Simrol, Madhya Pradesh 453552, India.} 
\email{dikshaba1233@gmail.com,  mscphd2207141001@iiti.ac.in}

\author{Bibekananda Maji}
\address{Bibekananda Maji\\ Discipline of Mathematics \\
Indian Institute of Technology Indore \\
Indore, Simrol, Madhya Pradesh 453552, India.} 
\email{bibek10iitb@gmail.com, bibekanandamaji@iiti.ac.in}


\begin{abstract}
In 1914, Hardy proved that there are infinitely many non-trivial zeros of the Riemann zeta function $\zeta(s)$ on the critical line $\Re(s)=1/2$ using the Jacobi theta relation. In this paper, we first establish a number field analogue of the Jacobi theta relation and as an application, we show the existence of infinitely many non-trivial zeros of the Dedekind zeta function $\zeta_\mathbb{F}(s)$ on $\Re(s)=1/2$, for any number field $\mathbb{F}$. Quite interestingly, we also prove that the Jacobi theta relation is equivalent to an intriguing identity of Hardy, Littlewood and Ramanujan.
\end{abstract}

\maketitle

\section{{\bf Introduction}}
The Riemann zeta function $\zeta(s)$ is regarded as one of the most important special functions in mathematics due to its connection with prime numbers and complex nature of non-trivial zeros.  Riemann \cite{Riemann},  in his seminal paper,  showed that $\zeta(s)$ can be analytically continued to the whole complex plane except for a simple pole at $s=1$ and proved the following beautiful symmetric functional equation:
\begin{align}\label{Functional eqn of zeta}
\pi^{-\frac{s}{2}}\Gamma\left(\frac{s}{2}\right)\zeta(s) = \pi^{-\frac{(1-s)}{2}}\Gamma\left(\frac{1-s}{2}\right)\zeta(1-s).
\end{align}
In the same paper,  he conjectured that all the non-trivial zeros of $\zeta(s)$ lie on the critical line $\Re(s)=1/2$,  which is popularly known as the Riemann hypothesis (RH).   
Riemann proved the functional equation \eqref{Functional eqn of zeta} using the following  Jacobi theta relation: 
\begin{align}\label{Jacobi theta transformation}
W_1\left(\frac{1}{x}\right) = \sqrt{x} W_1(x),
\end{align}
where $W_1(x)$ is the well-known Jacobi theta function, defined as 
\begin{align}\label{Jacobi theta function}
W_1(x) = 1 + 2\sum_{n = 1}^{\infty}e^{-\pi n^2 x}, \, \, \, \Re(x) > 0.
\end{align}
Moreover, one can prove that the above Jacobi theta relation \eqref{Jacobi theta transformation} is equivalent to the functional equation \eqref{Functional eqn of zeta} of $\zeta(s)$. 
Next, we present a similar relation formula due to Ramanujan \cite[p.~97]{GB},  \cite[p.~253]{Ramanujan}, namely, for $\Re(x) > 0$,
\begin{align}\label{Koshliakov Formula}
W_2\left(\frac{1}{x} \right) = \sqrt{x}W_2(x).
\end{align}
Here, $W_2(x)$ is given by
$$W_2(x) = \gamma - \log(4\pi) + \log(\sqrt{x}) + 4\sum_{n=1}^{\infty}d(n)\text{K}_0(2n\pi \sqrt{x}),$$
where $\gamma$ is the Euler-Mascheroni constant, $d(n)$ denotes the number of divisors of $n$, and $K_0$ is the modified Bessel function of the second kind. 
Besides Ramanujan, Koshliakov also recorded the above-mentioned identity \eqref{Koshliakov Formula} in his work \cite{Koshliakov}. Ferrar \cite{Ferrar1936} was the first person to prove the Ramanujan-Koshliakov identity \eqref{Koshliakov Formula} using the functional equation for $\zeta^2(s)$.  
More generally,  Ferrar \cite[p.~102]{Ferrar1936} derived a simultaneous extension of Jacobi theta relation and Ramanujan-Koshliakov formula by obtaining an identity under the assumption of the functional equation for $\zeta^k(s)$ for any $k \in \mathbb{N}$.  
It is worth to remark that, the functional equation of $\zeta^k(s)$ can be derived by assuming Ferrar's identity \cite[p.~102]{Ferrar1936}, see Corollary \ref{Equivalence for zeta power k}.

Motivated by Ferrar's identity \cite[Equation (2)]{Ferrar1936},  in this paper,   we establish an identity,  see Theorem \ref{BM1},  which is equivalent to the functional equation of $\zeta_{\mathbb{F}}^k(s)$,  where $\zeta_{\mathbb{F}}(s)$ is the Dedekind zeta function defined below.  

Let $\mathbb{F}$ be a number field of degree $d = r_1 + 2r_2$, where $r_1$ is the number of real embeddings and $r_2$ is the number of pairs of complex conjugate embeddings of $\mathbb{F}$. Let $D$ denote the absolute value of the discriminant of $\mathbb{F}$. Let $\mathcal{O}_{\mathbb{F}}$ denote the ring of integers of $\mathbb{F}$ and $\mathfrak{N}$ be the norm map. We define $\mathtt{a}_{\mathbb{F}}(n)$ to be the number of ideals in $\mathcal{O}_{\mathbb{F}}$ with norm $n$. Then the Dedekind zeta function corresponding to the number field $\mathbb{F}$ is defined by
\begin{align*}
\zeta_{\mathbb{F}}(s) = \sum_{\mathfrak{a} \subset \mathcal{O}_{\mathbb{F}}} \frac{1}{\mathfrak{N}(\mathfrak{a})^s} = \sum_{n=1}^{\infty}\frac{\mathtt{a}_{\mathbb{F}}(n)}{n^s}, 
\end{align*}
for $\Re{(s)} > 1,$ and the sum over $\mathfrak{a}$ runs through all non-zero integral ideals of $\mathcal{O}_{\mathbb{F}}$.
The Dedekind zeta function satisfies the following functional equation: 
\begin{align}\label{Dedekind_zeta_functional_equation}
\Omega_{\mathbb{F}}(s) = \Omega_{\mathbb{F}}(1-s),
\end{align}
where 
\begin{align}\label{Omega_F of fnal eqn}
\Omega_{\mathbb{F}}(s) &= \left(\frac{D}{4^{r_2}\pi^d}\right)^{\frac{s}{2}}\Gamma^{r_1}\left(\frac{s}{2}\right)\Gamma^{r_2}(s)\zeta_{\mathbb{F}}(s).
\end{align}
From \eqref{Dedekind_zeta_functional_equation}, we can also see that the functional equation for $\zeta_{\mathbb{F}}^{-1}(s)$ is given as
\begin{align*}
\Omega_{\mathbb{F}}^{-1}(s) = \Omega_{\mathbb{F}}^{-1}(1-s),
\end{align*}
which is equivalent to
\begin{align}\label{Functional eqn for dedekind zeta power -k}
\Lambda_{\mathbb{F}}(s) = \Lambda_{\mathbb{F}}(1-s),
\end{align}
where $\Lambda_{\mathbb{F}}(s) = \left(\frac{D}{4^{r_2}\pi^d}\right)^{\frac{s}{2}}\Gamma^{r_1}\left(\frac{s}{2}\right)\Gamma^{r_2}(s)\zeta^{-1}_{\mathbb{F}}(1-s)$.
It is known that $\zeta_{\mathbb{F}}(s)$ has a simple pole at $s=1$ with the residue given by the class number formula: 
\begin{align*}
\lim_{s \rightarrow 1}(s-1)\zeta_{\mathbb{F}}(s) = \frac{2^{r_1}(2\pi)^{r_2}}{\sqrt{D}}\frac{R_{\mathbb{F}}h_{\mathbb{F}}}{w_{\mathbb{F}} } :=H_{ \mathbb{F}},
\end{align*}
where $w_{\mathbb{F}}$ is the number of roots of unity contained in $\mathcal{O}_\mathbb{F}, h_{\mathbb{F}}$ denotes the class number, and $R_{\mathbb{F}}$ is the regulator of $\mathbb{F}$. 
Further,  we know that the order of zero of $\zeta_{\mathbb{F}}(s)$ at $s=0$ is given by $r=r_1+r_2-1$, which is also the rank of the unit group of $\mathbb{F}$. The leading term in the Laurent series expansion of $\zeta_{\mathbb{F}}(s)$ at $s=0$ is given by 
\begin{align}\label{Laurent series_at s=0_1st coeff}
\lim_{s\rightarrow 0} \frac{\zeta_{\mathbb{F}}(s)}{s^r} = - \frac{R_\mathbb{F} h_\mathbb{F}}{w_\mathbb{F}} :=C_{\mathbb{F}}.  
\end{align}
Quite interestingly, in this paper, we show that the functional equation for $\zeta^{-1}(s)$ is equivalent to the below well-known identity of Hardy, Littlewood and Ramanujan \cite[p.~158]{HLR1916}. 
Let $\mu(n)$ be the M\"{o}bius function. For any $x>0$, we have
\begin{align}\label{Ramanujan identity for 1/zeta}
\sum_{n=1}^\infty \frac{\mu (n)}{n}\exp\left(-\frac{x}{n^2}\right) = \sqrt{\frac{\pi}{x}}\sum_{n=1}^\infty \frac{\mu (n)}{n}\exp\left(-\frac{\pi^2}{n^2 x}\right) - \frac{1}{2\sqrt{\pi}}\sum_{\rho}\left(\frac{\pi}{\sqrt{x}}\right)^\rho \frac{\Gamma\left(\frac{1-\rho}{2}\right)}{\zeta'(\rho)},
\end{align}
where the sum over $\rho$ runs through the non-trivial zeros of $\zeta (s)$. 
This identity has been studied and generalized by many mathematicians, interested reader can see \cite{AGM22, BK23, Berndt4, Dixit12, Dixit15, Dixit16, GM23}. Recently, Dixit, Gupta and Vatwani \cite{DGV2022} extended the identity \eqref{Ramanujan identity for 1/zeta} for the Dedekind zeta function over any number field. 


One of the significant milestones in the pursuit of the Riemann hypothesis came through the work of Hardy \cite{Hardy}.  In 1914,  using the Jacobi theta relation \eqref{Jacobi theta transformation},  Hardy proved that there are infinitely many zeros of $\zeta(s)$ on the critical line $\Re(s) =\frac{1}{2}$. This result, while not proving the Riemann hypothesis itself, demonstrated the deep connection between the non-trivial zeros of $\zeta(s)$ and the critical line,  hinting at the possibility that all non-trivial zeros could indeed lie on this line. This result has been generalized for various Dirichlet series, with Hecke \cite{Hecke1937} being the first to show that the Dedekind zeta function has infinite number of zeros on the critical line for an ideal class in an imaginary quadratic field. Later, Chandrasekharan and  Narasimhan \cite{CN1968} extended this result to include both real and imaginary quadratic fields, while Berndt \cite{Berndt1} further extended it for any number field of degree $d = r_1 + 2r_2$ with $0 \leq r_1 \leq 3$. To learn more about zeros of Dedekind zeta function, readers are encouraged to see \cite{Berndt2, Berndt3, Brown77}. 

One of the main aims of the current paper is to develop an analogue of the Jacobi theta relation \eqref{Jacobi theta transformation} for  the Dedekind zeta function $ \zeta_{\mathbb{F}}(s)$  and employ Hardy's approach to prove the existence of infinitely many zeros of $ \zeta_{\mathbb{F}}(s)$  on the critical line $\Re(s)=1/2$.  
Before stating our main results, we need the following definitions.

The Steen function $V(z| a_1, a_2, \ldots, a_n)$ is defined as 
\begin{align*}
V(x| a_1, a_2, \ldots, a_n) := \frac{1}{2\pi i}\int_{(c)}\prod_{j=1}^{n}\Gamma(s + a_j)x^{-s}\text{d}s, \quad \text{for} \,\,\, |\Arg(x)| < \frac{\pi n}{2},
\end{align*}
where $(c)$ in the integral refers to the vertical line running from $c - i\infty$ to $c + i\infty$. Here, the argument $\Arg(x)$ is the principal argument $-\pi < \Arg(x) \leq \pi$.
It is assumed that all the poles of $\Gamma(s + a_j)$ lie on one side of the vertical line $(c)$.  This function is a special case of the more general Meijer $G$-function \cite[p.~415, Definition~16.17]{OLBC2010}. Particular choices of parameters recover various well-known special functions. For instance,
\begin{align}
V(x| 0) &= e^{-x}, \,\,\, \text{if} \,\,\, c > 0, \nonumber \\
V(x| a, b) &= 2x^{\frac{1}{2}(a+b)}\text{K}_{a-b}(2x^{\frac{1}{2}}), \,\,\, \text{if} \,\,\, c > \max\{-a, -b \}, \label{Bessel fn formula}
\end{align}
where $\text{K}_{\nu}$ is the modified Bessel function of second kind of order $\nu$. Further information about the Steen function can be found in \cite{Steen},  \cite[p.~63]{KT}.

We next examine a function, related to the Steen function, that was first studied by Koshliakov \cite[p.~122]{KoshI} and recently used by Dixit, Gupta, and Vatwani \cite[p.~5, Equation~(2.2)]{DGV2022}.

Let $c = \Re(s) > 0$. Then the function $\tilde{Z}_{r_1, r_2}: \mathbb{R} \rightarrow \mathbb{C}$ is defined as
\begin{align}\label{Defn of Z}
\tilde{Z}_{r_1, r_2}(x) := \frac{1}{2\pi i}\int_{(c)}\Gamma^{r_1}\left(\frac{s}{2}\right)\Gamma^{r_2}(s)x^{-s} \text{d}s, \quad \text{where} \,\,\, |\Arg(x)| < \frac{\pi d}{4}.
\end{align}
If we shift the line of integration to the left, we get the following representation, for $-1 < \Re(s) = b <0$,
\begin{align}\label{Defn of Z without tilde}
{Z}_{r_1, r_2}(x) :=  \frac{1}{2\pi i}\int_{(b)}\Gamma^{r_1}\left(\frac{s}{2}\right)\Gamma^{r_2}(s)x^{-s} \mathrm{d}s =  \tilde{Z}_{r_1, r_2}(x) - \mathfrak{R}_0(x), 
\end{align}
where $ \mathfrak{R}_0(x)$ is the residue of the above integrand function due to a pole of order $r_1 + r_2$ at $s = 0$. 
Some particular cases can be given by
\begin{align}
\tilde{Z}_{0, r_2}(x) &= V(x| \bar{0}_{r_2}), \label{Z value img}\\ 
\tilde{Z}_{r_1, 0}(x) &= 2V(x^2| \bar{0}_{r_1}), \label{Z value real} \\
{Z}_{r_1, 0}(x) &= 2V(x^2|\bar{0}_{r_1}) - \mathfrak{R}_0(x), \label{Z without tilde at real} \\
{Z}_{1, 0}(x) &= 2(e^{-x^2} - 1), \label{Z without tilde value at Q}
\end{align}
where $\bar{0}_{m}$ is the $m$-tuple whose all entries equal 0, and $\mathfrak{R}_0(x)$ in \eqref{Z without tilde at real} is given by
\begin{align*}
\mathfrak{R}_0(x) = \frac{1}{(r_1 - 1)!}\lim_{s \rightarrow 0}\frac{\mathrm{d}^{r_1 -1}}{\mathrm{d}s^{r_1-1}}\left(s^{r_1} \Gamma^{r_1}\left(\frac{s}{2}\right)x^{-s}\right).
\end{align*}
\begin{remark}
Using the duplication formula for $\Gamma(s)$, one can check that 
\begin{align*}
\tilde{Z}_{r_1, r_2}(x) &= \frac{2}{2^{r_2}\pi^{\frac{r_2}{2}}}V\left(\frac{x^2}{4^{r_2}} \bigg| \bar{0}_{r_1 +  r_2}, \bar{\left(\frac{1}{2}\right)}_{r_2}\right), \quad \text{for} \,\,\,\, |\Arg(x)| < \frac{\pi d}{4}.
  \end{align*}
\end{remark}
Before moving further, we define the following Dirichlet series for $\Re(s) > 1$:
\begin{align}
\zeta_{\mathbb{F}}^k(s) := \sum_{n=1}^{\infty}\frac{\mathtt{a}_{\mathbb{F}, k}(n)}{n^s}, \quad
\zeta^k(s) := \sum_{n=1}^{\infty}\frac{d_k(n)}{n^s}, \label{Coeff a and d_k}
\end{align}
where $\mathtt{a}_{\mathbb{F}, k}(n) = \mathtt{a}_{\mathbb{F}}(n)*\cdots*\mathtt{a}_{\mathbb{F}}(n)$ i.e. Dirichlet convolution of $\mathtt{a}_{\mathbb{F}}(n),\, k-$times and $d_k(n)$ counts the number of ways $n$ can be written as product of $k$ positive integers. Hence, for $k=2$, we have $d_k(n) = d(n)$.
We now state the main results of this paper.
                   
\section{{\bf Main Results}}

\begin{theorem}\label{BM1}
Let $k$ be a positive integer and $\mathbb{F}$ be any number field of degree $d$. Let $\Omega_{\mathbb{F}}(s)$, $\tilde{Z}_{r_1, r_2}(x)$ and $\mathtt{a}_{\mathbb{F}, k}(n)$ be defined as in \eqref{Omega_F of fnal eqn}, \eqref{Defn of Z} and \eqref{Coeff a and d_k}, respectively. The functional equation for $\zeta_\mathbb{F}^k(s)$ 
is equivalent to the theta relation 
\begin{align}\label{Theta reln for dedekind zeta power k}
W_{\mathbb{F}, k}\left(\frac{1}{x}\right) = \sqrt{x}W_{\mathbb{F}, k}(x), \quad \text{for} \,\,\, |\Arg(x)| < \frac{\pi d}{2},
\end{align}
where $W_{\mathbb{F}, k}(x)$ is given by
\begin{align}\label{W_F, k(x)}
W_{\mathbb{F}, k}(x) := \sum_{n=1}^{\infty}\mathtt{a}_{\mathbb{F}, k}(n)\tilde{Z}_{kr_1, kr_2}\left(\frac{2^{kr_2}\pi^{kd/2} n\sqrt{x}}{\sqrt{D^k}} \right) - R_0(x),
\end{align}
with
 $$
 R_0(x) = \frac{1}{(k-1)!}\lim_{s \rightarrow 0} \frac{\mathrm{d}^{k-1}}{\mathrm{d}s^{k-1}}\left(s^k \Omega_{\mathbb{F}}^k(s)x^{-s/2}\right).
 $$
\end{theorem}

\begin{remark}
For $d=1$, that is, $\mathbb{F} = \mathbb{Q}$, the relation \eqref{Theta reln for dedekind zeta power k} holds for all $x$ with $\Re(x) > 0$. For quadratic fields, it is valid for all $x \in \mathbb{C} \setminus (-\infty, 0]$, whereas if the field $\mathbb{F}$ has degree $d \geq 3$, the relation \eqref{Theta reln for dedekind zeta power k} holds for any complex number $x \in \mathbb{C} \setminus \{0\}$.
\end{remark}

In particular, if we take $k=1$, we get the following result for the Dedekind zeta function.
\begin{corollary}\label{Jacobi theta for dedekind zeta}
Let $\mathbb{F}$ be a number field with degree $d$. Then the functional equation for $\zeta_{\mathbb{F}}(s)$ is equivalent to the theta relation 
\begin{align}\label{Theta reln for dedekind zeta}
W_{\mathbb{F}, 1}\left(\frac{1}{x}\right) = \sqrt{x}W_{\mathbb{F}, 1}(x), \quad \text{for} \,\,\, |\Arg(x)| < \frac{\pi d}{2},
\end{align}
where $W_{\mathbb{F}, 1}(x)$ is given by 
\begin{align}\label{W_F,1(x)}
W_{\mathbb{F}, 1}(x) = \sum_{n=1}^{\infty}\mathtt{a}_{\mathbb{F}}(n)\tilde{Z}_{r_1, r_2}\left(\frac{2^{r_2}\pi^{d/2} n\sqrt{x}}{\sqrt{D}} \right) - 2^{r_1}C_{\mathbb{F}},
\end{align}
where $C_{\mathbb{F}}$ is defined in \eqref{Laurent series_at s=0_1st coeff}. 
\end{corollary}
\begin{remark}
We note that $W_{\mathbb{F}, 1}(x)$ serves as a number field analogue of the Jacobi theta function $W_1(x)$, defined in \eqref{Jacobi theta function}, as it coincides with $W_1(x)$ when $\mathbb{F} = \mathbb{Q}$. For number fields of degree $d \leq 2$, the theta relation \eqref{Theta reln for dedekind zeta} has also been studied by Arai, Chakraborty and Kanemitsu \cite[Equation~(2.11)]{ACK2015}.
\end{remark}
\begin{remark}
For any number field $\mathbb{F}$ of degree $d \geq 3$, letting $x= -1$ in \eqref{Theta reln for dedekind zeta}, we derive the following exact evaluation:
\begin{align*}
\sum_{n=1}^{\infty}\mathtt{a}_{\mathbb{F}}(n)\tilde{Z}_{r_1, r_2}\left(\frac{2^{r_2}\pi^{d/2} ni}{\sqrt{D}} \right) &= 2^{r_1}C_{\mathbb{F}}.
\end{align*}
\end{remark}

We next present an equivalence for the functional equation of the Dedekind zeta function over totally real number fields.
\begin{corollary}\label{Theta relation for real zeta power k}
Let $\mathbb{F}$ be a totally real number field with degree $d = r_1$. The functional equation for $\zeta^k_{\mathbb{F}}(s)$ 
is equivalent to the theta relation 
\begin{align*}
W_{\mathbb{F}, k}\left(\frac{1}{x}\right) = \sqrt{x}W_{\mathbb{F}, k}(x), \quad \text{for} \,\,\, |\Arg(x)| < \frac{\pi r_1}{2},
\end{align*}
where $W_{\mathbb{F}, k}(x)$ is given by 
$$
W_{\mathbb{F}, k}(x) = 2\sum_{n=1}^{\infty}\mathtt{a}_{\mathbb{F}, k}(n)V\left(\frac{\pi^{kr_1} n^2 x}{D^k} \bigg| \bar{0}_{kr_1} \right) - R_0(x),
$$ 
and 
$$
R_0(x) = \frac{1}{(k-1)!}\lim_{s \rightarrow 0}\frac{\mathrm{d}^{k-1}}{\mathrm{d}s^{k-1}}\left(s^k \Gamma^{kr_1}\left(\frac{s}{2}\right)\zeta_\mathbb{F}^k(s)\left(\frac{\pi^{(kr_{1}/2)} \sqrt{x}}{\sqrt{D^k}} \right)^{-s}\right).
$$
\end{corollary}

%

Now assuming $\mathbb{F}$ to be a real quadratic field,  we obtain the following result.
\begin{corollary}\label{Theta relation for real quadratic field}
Let $\mathbb{F} = \mathbb{Q}(\sqrt{m})$, where $m$ is a positive square free integer. The functional equation for $\zeta_{\mathbb{F}}(s)$ 
is equivalent to the theta relation 
\begin{align*}
W_{\mathbb{F}, 1}\left(\frac{1}{x}\right) = \sqrt{x}W_{\mathbb{F}, 1}(x), \quad \text{for} \,\,\, |\Arg(x)| < \pi,
\end{align*}
where $W_{\mathbb{F}, 1}(x)$ is given by $$W_{\mathbb{F}, 1}(x) = 4\sum_{n=1}^{\infty}\mathtt{a}_{\mathbb{F}}(n)K_0\left(\frac{2\pi n \sqrt{x}}{\sqrt{D}}\right) - 4C_{\mathbb{F}}.$$
\end{corollary}
We obtain Ferrar's identity \cite[p.~102]{Ferrar1936} by taking $\mathbb{F} = \mathbb{Q}$ in Corollary \ref{Theta relation for real zeta power k}.

\begin{corollary}\label{Equivalence for zeta power k}
The functional equation for $\zeta^k(s)$ is equivalent
to the theta relation 
\begin{align}\label{Theta reln for zeta power k}
W_{\mathbb{Q}, k}\left(\frac{1}{x}\right) = \sqrt{x}W_{\mathbb{Q}, k}(x), \quad \text{for} \,\,\, |\Arg(x)| < \frac{\pi}{2},
\end{align}
where $W_{\mathbb{Q}, k}(x)$ is given by 
$$
W_{\mathbb{Q}, k}(x) := 2\sum_{n=1}^{\infty}d_k(n)V\left(\pi^k n^2 x | \bar{0}_k \right) - R_0(x),
$$ 
and 
$$
R_0(x) = \frac{1}{(k-1)!}\lim_{s \rightarrow 0}\frac{\mathrm{d}^{k-1}}{\mathrm{d}s^{k-1}}\left(s^k \Gamma^k\left(\frac{s}{2}\right)\zeta^k(s)\left(\pi^k x\right)^{-\frac{s}{2}}\right).
$$
\end{corollary}
\begin{remark}
Ferrar \cite[p.~102]{Ferrar1936} derived the identity \eqref{Theta reln for zeta power k} by assuming the functional equation for $\zeta^k(s)$. However, he did not prove the converse, that is, the relation \eqref{Theta reln for zeta power k} implies the functional equation for $\zeta^k(s)$.
\end{remark}

\begin{remark}
Putting $k = 1$ and $k = 2$ in Corollary \ref{Equivalence for zeta power k} gives the equivalences for functional equation of $\zeta(s)$ and $\zeta^2(s)$ with the Jacobi theta relation \eqref{Jacobi theta transformation} and the Ramanujan-Koshliakov fomula \eqref{Koshliakov Formula}, respectively.
\end{remark}

\subsection{Equivalence of the functional equation of $\zeta_\mathbb{F}^{-k}(s)$}
In this subsection, we state an equivalence of the functional equation for $\zeta_\mathbb{F}^{-k}(s)$ with $k \in \mathbb{N}$. It is well-known that $\zeta_\mathbb{F}(s) \ne 0$ for $\Re(s) > 1$. So, we define a generalization of the M\"obius function $\mu(n)$ for any number field $\mathbb{F}$ using the following Dirichlet series:
\begin{align}\label{Gen fun for 1/Dzeta}
\frac{1}{\zeta_\mathbb{F}^k(s)} = \sum_{n=1}^{\infty}\frac{\mu_{\mathbb{F}, k}(n)}{n^s}, \quad \Re(s) > 1.
\end{align}
A detailed explanation for $\mu_{\mathbb{F}, k}(n)$ is given in \eqref{Defn of Mu_F}.
\begin{theorem}\label{BM3}
Let $k$ be any positive integer, $r = r_1 + r_2 -1$, and $\Lambda_{\mathbb{F}}(s)$, $Z_{r_1, r_2}(x)$ be defined as in \eqref{Functional eqn for dedekind zeta power -k}, \eqref{Defn of Z without tilde}, respectively. Then the functional equation for $\zeta_\mathbb{F}^{-k}(s)$ 
is equivalent to the theta relation
\begin{align}\label{Theta reln for dedekind zeta power -k}
U_{\mathbb{F}, -k}\left(\frac{1}{x}\right) = \sqrt{x}U_{\mathbb{F}, -k}(x), \quad \text{for} \,\,\, |\Arg(x)| < \frac{\pi d}{2},
\end{align}
where the term $U_{\mathbb{F}, -k}(x)$ is given by
\begin{align}\label{U_F, k(x)}
U_{\mathbb{F}, -k}(x) &:= \sum_{n=1}^{\infty}\frac{\mu_{\mathbb{F}, k}(n)}{n} {Z}_{kr_1, kr_2}\left(\frac{2^{kr_2}\pi^{kd/2} \sqrt{x}}{n\sqrt{D^k}} \right) 
 + R_0(x) + \frac{1}{2}\sum_{\rho} R_{\rho}(x),
\end{align}
and the terms $R_0(x)$ and $R_\rho(x)$ are defined as
\begin{align*}
R_0(x) &= \frac{1}{(kr-1)!}\lim_{s \rightarrow 0}\frac{{\rm d}^{kr-1}}{{\rm d}s^{kr-1}}\left(s^{kr} \Lambda_{\mathbb{F}}^{k}(s)( \sqrt{x})^{-s}\right),\\
R_\rho(x) &= \frac{1}{(k-1)!}\lim_{s \rightarrow \rho}\frac{{\rm d}^{k-1}}{{\rm d}s^{k-1}}\left((s-\rho)^k \Lambda_{\mathbb{F}}^{k}(s)(\sqrt{x})^{-s}\right).
\end{align*}
Here, we assume that, for any $x>0$, the infinite series $\sum_{\rho} R_{\rho}(x)$ converges, where the sum runs over all the non-trivial zeros of $\zeta_{\mathbb{F}}(s)$ in the strip $0 < \Re(s) < 1$. Additionally, we suppose that each such zero is simple. 
\end{theorem}
\begin{remark}
Setting $k=1, \alpha = \frac{2^{r_2}\pi^{d/2}\sqrt{x}}{\sqrt{D}}$ and $\beta = \frac{2^{r_2}\pi^{d/2}}{\sqrt{D}\sqrt{x}}$ in Theorem \ref{BM3} yields that the functional equation of $\zeta_\mathbb{F}^{-1}(s)$ is equivalent to the following identity of Dixit, Gupta, and Vatwani \cite[Theorem~1.3]{DGV2022},
\begin{align}\label{Id of Dixit etal}
&\sqrt{\alpha}\sum_{n=1}^{\infty}\frac{\mu_{\mathbb{F}, 1}(n)}{n}{Z}_{r_1, r_2}\left(\frac{\alpha}{n} \right) - \sqrt{\beta}\sum_{n=1}^{\infty}\frac{\mu_{\mathbb{F}, 1}(n)}{n}{Z}_{r_1, r_2}\left(\frac{\beta}{n} \right) \nonumber \\
&= \frac{1}{\sqrt{\alpha}}R_0\left(\alpha\right) - \frac{1}{\sqrt{\beta}}R_0\left(\beta\right) + \frac{1}{2}\left[\frac{1}{\sqrt{\alpha}}\sum_\rho R_\rho\left(\alpha\right) - \frac{1}{\sqrt{\beta}}\sum_\rho R_\rho\left(\beta\right)\right],
\end{align}
where the terms $R_0(\alpha)$ and $R_\rho(\alpha)$ are given by
\begin{align*}
R_0(\alpha) &= \frac{1}{(r-1)!}\lim_{s \rightarrow 0}\frac{{\rm d}^{r-1}}{{\rm d}s^{r-1}}\left(s^r \alpha^s \frac{\Gamma^{r_1}\left(\frac{1-s}{2}\right)\Gamma^{r_2}\left(1-s\right)}{\zeta_\mathbb{F}(s)} \right), \\
R_\rho\left(\alpha\right) &= \alpha^\rho \frac{\Gamma^{r_1}\left(\frac{1-\rho}{2}\right)\Gamma^{r_2}\left(1-\rho\right)}{\zeta_\mathbb{F}'(\rho)}.
\end{align*}
This also suggests that the theta relation \eqref{Theta reln for dedekind zeta power -k} is a one variable generalization of the identity \eqref{Id of Dixit etal} of Dixit, Gupta, and Vatwani.

\end{remark}

Considering $\mathbb{F} = \mathbb{Q}$ in Theorem \ref{BM3}, we have the following result.
\begin{corollary}\label{Theta reln for zeta power -k}
Let $k \in \mathbb{N}$ and $\mu_k(n)$ be the generalized M\"obius  function defined in \eqref{Mobius for Q}. The functional equation for $\zeta^{-k}(s)$
is equivalent to the theta relation 
\begin{align}\label{New id for zeta inv}
U_{\mathbb{Q}, -k}\left(\frac{1}{x}\right) = \sqrt{x}U_{\mathbb{Q}, -k}(x), \quad \text{for} \,\,\, |\Arg(x)| < \frac{\pi}{2},
\end{align}
where $U_{\mathbb{Q}, -k}(x)$ is given by 
\begin{align*}
U_{\mathbb{Q}, -k}(x) := \sum_{n=1}^{\infty}\frac{\mu_{k}(n)}{n}\left[2\,V\left(\frac{\pi^{k} x}{n^2} \bigg | \bar{0}_{k} \right) - \mathfrak{R}_0\left(\frac{\pi^{k/2} \sqrt{x}}{n}\right)\right] + \frac{1}{2}\sum_{\rho} R_{\rho}(x).
\end{align*} 
Here, the sum over $\rho$ runs through all the non-trivial zeros of $\zeta(s)$, and
{\allowdisplaybreaks
\begin{align*}
\mathfrak{R}_0(x) &= \frac{1}{(k-1)!}\lim_{s \rightarrow 0}\frac{\mathrm{d}^{k -1}}{\mathrm{d}s^{k-1}}\left(s^{k} \Gamma^{k}\left(\frac{s}{2}\right)x^{-s}\right), \\
R_\rho(x) &= \frac{1}{(k-1)!}\lim_{s \rightarrow \rho}\frac{{\rm d}^{k-1}}{{\rm d}s^{k-1}}\left((s-\rho)^k \pi^\frac{-ks}{2}\Gamma^k\left(\frac{s}{2}\right)\zeta^{-k}(1-s)(\sqrt{x})^{-s}\right).
\end{align*}}
\end{corollary}

\begin{remark} 
Upon simplification, substituting $k=1$ in Corollary \ref{Theta reln for zeta power -k}, one can deduce that the functional equation of $\zeta^{-1}(s)$ is equivalent to the Hardy-Littlewood-Ramanujan identity \eqref{Ramanujan identity for 1/zeta}. Thus, the identity \eqref{New id for zeta inv} can be seen as a one variable generalization of \eqref{Ramanujan identity for 1/zeta}.
\end{remark}

\begin{corollary}\label{Theta reln for zeta power -2}
Let $\mu_2(n)$ be the generalized M\"obius  function defined in \eqref{Mobius for Q}. The functional equation for $\zeta^{-2}(s)$
is equivalent to the theta relation 
\begin{align}\label{New id for zeta 2inv}
U_{\mathbb{Q}, -2}\left(\frac{1}{x}\right) = \sqrt{x}U_{\mathbb{Q}, -2}(x), \quad \text{for} \,\,\, |\Arg(x)| < \frac{\pi}{2},
\end{align}
where $U_{\mathbb{Q}, -2}(x)$ is given by 
\begin{align*}
U_{\mathbb{Q}, -2}(x) := 4\sum_{n=1}^{\infty}\frac{\mu_{2}(n)}{n}\left[\,K_0\left(\frac{2\pi \sqrt{x}}{n}\right) + \log\left(\frac{\pi \sqrt{x}}{n}\right)\right] + \frac{1}{2}\sum_{\rho} R_{\rho}(x).
\end{align*} 
Here, the sum over $\rho$ runs through all the non-trivial zeros of $\zeta(s)$, and
\begin{align*}
R_\rho(x) &= \lim_{s \rightarrow \rho}\frac{{\rm d}}{{\rm d}s}\left((s-\rho)^2 \Gamma^2\left(\frac{s}{2}\right)\zeta^{-2}(1-s)(\pi\sqrt{x})^{-s}\right).
\end{align*}
\end{corollary}

Now we make use of Corollary \ref{Jacobi theta for dedekind zeta} to prove an analogue of Hardy's result for the Dedekind zeta function.  
\begin{theorem}\label{BM2}
For any number field $\mathbb{F}$, the Dedekind zeta function $\zeta_{\mathbb{F}}(s)$ has infinitely many non-trivial zeros on the critical line $\Re(s) = \frac{1}{2}$.
\end{theorem}

\section{{\bf Key Tools}}\label{Key Tools}
In this section, we mention some results and lemmas that will be used in proving our main theorems. 
The Dedekind zeta function $\zeta_{\mathbb{F}}(s)$ has Euler product expansion, given by
$$ \zeta_\mathbb{F}(s) = \prod_{\mathfrak{p} \subset \mathcal{O}_\mathbb{F}}\left(1 - \frac{1}{\mathfrak{N}(\mathfrak{p})^s}\right)^{-1}, \quad \Re(s) > 1,$$
where $\mathfrak{p}$ runs over prime ideals of $\mathcal{O}_{\mathbb{F}}$. From \eqref{Gen fun for 1/Dzeta}, we know that the Dirichlet series associated to $\mu_{\mathbb{F}, k}(n)$ is $\zeta_\mathbb{F}^{-k}(s)$. Thus, using the above Euler product, one can see that 
\begin{align}\label{Defn of Mu_F}
\mu_{\mathbb{F}, k}(n) = \begin{cases} 
       \sum\limits_{\mathfrak{a} \in A}(-1)^{\alpha_1 + \alpha_2 + \cdots + \alpha_m}, & 1 \leq \alpha_i \leq k, \\
      1, & n = 1, \\
      0, & \text{otherwise}.
   \end{cases}
\end{align}
where the set $A$ is defined as 
$$ A = \{ \mathfrak{a} \in \mathcal{O}_\mathbb{F}| \mathfrak{N}(\mathfrak{a}) = n, \mathfrak{a} = \mathfrak{p}_1^{\alpha_1}\mathfrak{p}_2^{\alpha_2}\cdots\mathfrak{p}_m^{\alpha_m}, \text{where } \mathfrak{p}_i\text{'s} \text{ are distinct prime ideals}\}.
$$
For $k = 1$, the above function $\mu_{\mathbb{F}, k}(n)$ coincides with the function given by Heath-Brown \cite[p.~172]{Brown77}.
In particular, when $\mathbb{F} = \mathbb{Q}$, then
\begin{align}\label{Mobius for Q}
\mu_{\mathbb{Q}, k}(n) = \mu_{k}(n) = \begin{cases} 
       (-1)^{\alpha_1 + \alpha_2 + \cdots + \alpha_m}, & 1 \leq \alpha_i \leq k, \\
      1, & n = 1, \\
      0, & \text{otherwise},
   \end{cases}
\end{align}
where $n = p_1^{\alpha_1}p_2^{\alpha_2}\cdots p_m^{\alpha_m}$ and $\mu_{1}(n) = \mu(n)$, with $\mu(n)$ being the well-known M\"obius function. Moreover, Landau \cite[p.~172]{Landau1903} showed that
\begin{align*}
\lim_{s \rightarrow 1^+}\frac{1}{\zeta_\mathbb{F}(s)} = \lim_{s \rightarrow 1^+}\sum_{n=1}^{\infty}\frac{\mu_{\mathbb{F}}(n)}{n^s} = \sum_{n=1}^{\infty}\frac{\mu_{\mathbb{F}}(n)}{n} = 0.
\end{align*}
More generally, for $k \geq 1$, one can see that
\begin{align}\label{Sum of 1/Dzeta(1)}
\lim_{s \rightarrow 1^+}\frac{1}{\zeta_\mathbb{F}^k(s)} = \lim_{s \rightarrow 1^+}\sum_{n=1}^{\infty}\frac{\mu_{\mathbb{F}, k}(n)}{n^s} = \sum_{n=1}^{\infty}\frac{\mu_{\mathbb{F}, k}(n)}{n} = 0.
\end{align}
We now define the following entire function:
\begin{align}\label{xi function}
\xi_{\mathbb{F}}(s) &:= \frac{1}{2}s(s-1) \left(\frac{D}{4^{r_2}\pi^d}\right)^{\frac{s}{2}}\Gamma^{r_1}\left(\frac{s}{2}\right)\Gamma^{r_2}(s)\zeta_{\mathbb{F}}(s) \nonumber \\
&= \frac{1}{2}s(s-1)\Omega_{\mathbb{F}}(s) = \xi_{\mathbb{F}}(1-s).
\end{align}
Putting $s=\frac{1}{2}+ it$ in the above equation and denoting it as $\Xi_{\mathbb{F}}(t)$, we have
\begin{align*}
\xi_{\mathbb{F}}\left(\frac{1}{2}+it\right) = \Xi_{\mathbb{F}}(t):= \left(-\frac{1}{8} - \frac{t^2}{2} \right)\left(\frac{D}{4^{r_2}\pi^d}\right)^{\frac{1}{4}+\frac{it}{2}}\Gamma^{r_1}\left(\frac{1}{4} + \frac{it}{2}\right)\Gamma^{r_2}\left(\frac{1}{2} + it\right)\zeta_{\mathbb{F}}\left(\frac{1}{2} + it\right).
\end{align*}
We prove the below result concerning $\Xi_{\mathbb{F}}(t)$ which will be used in proving Theorem \ref{BM2}.
\begin{lemma}\label{Xi as even and real fn}
The function $\Xi_{\mathbb{F}}(t)$ is even and real for all real values of t.
\end{lemma}
\begin{proof}
Consider
{\allowdisplaybreaks
\begin{align*}
\overline{\Xi_{\mathbb{F}}(t)} &= \overline{\left(-\frac{1}{8} - \frac{t^2}{2} \right)}\overline{\left(\frac{D}{4^{r_2}\pi^d}\right)^{\frac{1}{4}+\frac{it}{2}}}\overline{\Gamma^{r_1}\left(\frac{1}{4} + \frac{it}{2}\right)}\overline{\Gamma^{r_2}\left(\frac{1}{2} + it\right)}\overline{\zeta_\mathbb{F}\left(\frac{1}{2} + it\right)} \nonumber \\
&= \left(-\frac{1}{8} - \frac{t^2}{2} \right)\left(\frac{D}{4^{r_2}\pi^d}\right)^{\frac{1}{4}-\frac{it}{2}}\Gamma^{r_1}\left(\frac{1}{4} - \frac{it}{2}\right)\Gamma^{r_2}\left(\frac{1}{2} - it\right)\zeta_{\mathbb{F}}\left(\frac{1}{2} - it\right) \nonumber \\
&= \Xi_{\mathbb{F}}(-t) \nonumber \\
&= \xi_{\mathbb{F}}\left(\frac{1}{2} - it\right) \nonumber \\
&= \xi_{\mathbb{F}}\left(1 - \frac{1}{2} + it\right) \nonumber \\
&= \xi_{\mathbb{F}}\left(\frac{1}{2} + it\right) \nonumber \\
&= \Xi_{\mathbb{F}}(t),
\end{align*}}
since $\overline{\Gamma(s)} = \Gamma(\overline{s})$ and $\overline{\zeta_\mathbb{F}(s)} = \zeta_\mathbb{F}(\overline{s})$. In the penultimate step, we have used the fact that $\xi_{\mathbb{F}}(s) = \xi_\mathbb{F}(1 - s)$. This completes the proof.
\end{proof}

Now we mention the following exact evaluation of an integral \cite[p.~36]{Titchmarsh}, where the integrand is related to the functional equation of the Riemann zeta function, namely,
\begin{align*}
\frac{1}{2\pi i}\int_{\frac{1}{2} - i\infty}^{\frac{1}{2} + i\infty}\pi^{-\frac{s}{2}}\Gamma\left(\frac{s}{2}\right)\zeta(s)y^s \text{d}s = W_1\left(\frac{1}{y^2}\right) - (1 + y),
\end{align*}
where $W_{1}(x)$ is the Jacobi theta function \eqref{Jacobi theta function}. We generalize the above integral evaluation for the Dedekind zeta function.

\begin{lemma}\label{Lemma 1}
Let $\Omega_{\mathbb{F}}(s)$ and $W_{\mathbb{F},1}(x)$ be the functions defined as in \eqref{Omega_F of fnal eqn} and \eqref{W_F,1(x)}, respectively. Then we have
\begin{align}\label{lemma 1}
\frac{1}{2\pi i}\int_{\frac{1}{2} - i\infty}^{\frac{1}{2} + i\infty}\Omega_{\mathbb{F}}(s)y^s \text{d}s = W_{\mathbb{F},1}\left(\frac{1}{y^2}\right) + C_\mathbb{F}2^{r_1}(1+y).
\end{align}
\end{lemma}	
\begin{proof}
Using the definition \eqref{W_F,1(x)} of $W_{\mathbb{F},1}(x)$ and \eqref{Defn of Z} of $\tilde{Z}_{r_1, r_2}(x)$, one can write
\begin{align*}
W_{\mathbb{F},1}\left(\frac{1}{y^2}\right) &= \sum_{n=1}^\infty \mathtt{a}_\mathbb{F}(n)\tilde{Z}_{r_1, r_2}\left(\frac{2^{r_2}\pi^{d/2} n}{y\sqrt{D}} \right) - C_\mathbb{F}2^{r_1}  \nonumber\\
&= \sum_{n=1}^\infty \mathtt{a}_\mathbb{F}(n) \frac{1}{2\pi i} \int_{c - i\infty}^{c+i\infty}\Gamma^{r_1}\left(\frac{s}{2}\right)\Gamma^{r_2}(s)\left(\frac{2^{r_2}\pi^{d/2} n}{y\sqrt{D}} \right)^{-s} \text{d}s - C_\mathbb{F}2^{r_1},
\end{align*}
with $c>1$. Hence, upon changing the summation and integration, we get
\begin{align}
\frac{1}{2\pi i} \int_{c - i\infty}^{c+i\infty}\Omega_{\mathbb{F}}(s)y^{s}\text{d}s &= C_\mathbb{F}2^{r_1} + W_{\mathbb{F},1}\left(\frac{1}{y^2}\right), \label{integral with fnal eqn}
\end{align}
where $\Omega_{\mathbb{F}}(s)$ is defined in \eqref{Omega_F of fnal eqn}.
We now set up a rectangular contour $\mathcal{C}$ made up of the vertices $c-iT, c+iT, \frac{1}{2}+iT$ and $\frac{1}{2}-iT$ taken in counter-clockwise direction with $T$ being some large positive quantity.  
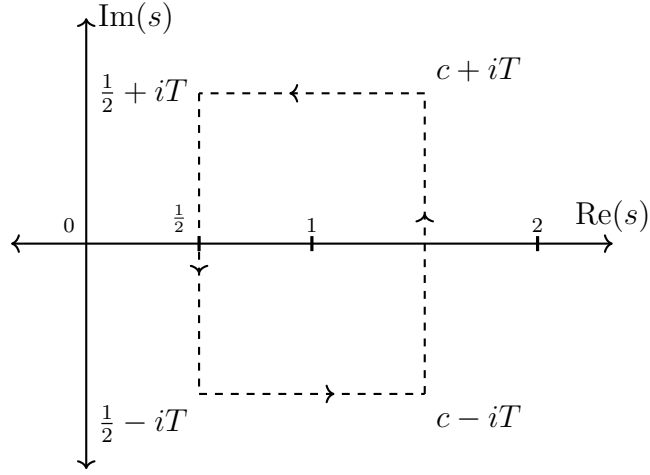
\begin{center}
\begin{tikzpicture}[very thick,decoration={
  markings,
  mark=at position 0.6 with {\arrow{>}}}
 ] 
 \draw[thick,dashed,postaction={decorate}] (-0.5,-2)--(2.5,-2) node[below right, black] {$c-i T$};
 \draw[thick,dashed,postaction={decorate}] (2.5,-2)--(2.5,2) node[above right, black] {$c+iT$} ;
 \draw[thick,dashed,postaction={decorate}] (2.5,2)--(-0.5,2) node[left, black] {$\frac{1}{2}+i T$}; 
 \draw[thick,dashed,postaction={decorate}] (-0.5,2)--(-0.5,-2) node[below left, black] {$\frac{1}{2}-i T$}; 
 \draw[thick, <->] (-3,0) -- (5,0) coordinate (xaxis);
 \draw[thick, <->] (-2,-3) -- (-2,3)node[midway, above left, black] {\tiny0} coordinate (yaxis);
 \draw (-0.5,0.1)--(-0.5,-0.1) node[midway, above left, black]{\tiny $\frac{1}{2}$};
 \draw (1,0.1)--(1,-0.1) node[midway, above, black]{\tiny $1$};
 \draw (4,0.1)--(4,-0.1) node[midway, above, black]{\tiny $2$};
 \node[above] at (xaxis) {$\mathrm{Re}(s)$};
 \node[right] at (yaxis) {$\mathrm{Im}(s)$};
\end{tikzpicture}
\end{center}
Now we use Cauchy's residue theorem to have
\begin{align*}
\frac{1}{2\pi i}\int_{\mathcal{C}}\Omega_{\mathbb{F}}(s)y^{s}\text{d}s = \mathcal{R},
\end{align*}
where $\mathcal{R}$ is the sum of the residues of the integrand function inside the contour. We know that $\zeta_\mathbb{F}(s)$ has a simple pole at $s=1$ with residue $H_{\mathbb{F}}$ which indicates that the integrand function has a simple pole at $s=1$ inside the contour. Letting $T \rightarrow \infty$, we can see that the horizontal integrals vanish due to the exponential decay of $\Gamma(s)$. Therefore, we are left with
\begin{align}
\frac{1}{2\pi i} \int_{c - i\infty}^{c+i\infty}\Omega_{\mathbb{F}}(s)y^{s}\text{d}s &= \frac{1}{2\pi i} \int_{\frac{1}{2} - i\infty}^{\frac{1}{2}+i\infty}\Omega_{\mathbb{F}}(s)y^{s}\text{d}s + \mathcal{R} \nonumber \\
\Rightarrow \frac{1}{2\pi i} \int_{\frac{1}{2} - i\infty}^{\frac{1}{2}+i\infty}\Omega_{\mathbb{F}}(s)y^{s}\text{d}s &= C_\mathbb{F}2^{r_1} + W_{\mathbb{F},1}\left(\frac{1}{y^2}\right) - \mathcal{R}. \label{lemma 1 with R}
\end{align}
Here we have used \eqref{integral with fnal eqn} in the last step. Now we calculate the residual term,
\begin{align*}
\mathcal{R} &= \lim_{s \rightarrow 1} (s-1)\left(\frac{D}{4^{r_2}\pi^d}\right)^{\frac{s}{2}}\Gamma^{r_1}\left(\frac{s}{2}\right)\Gamma^{r_2}(s)\zeta_{\mathbb{F}}(s)y^s\\
&= -2^{r_1}yC_\mathbb{F}.
\end{align*}
Using the value of $\mathcal{R}$ in \eqref{lemma 1 with R}, we complete the proof of \eqref{lemma 1}.
\end{proof}

In the next section, we present proofs of the main results.

\section{{\bf Proof of Main Results}}

\begin{proof}[Theorem \rm{\ref{BM1}}][]
Let us define 
\begin{align}\label{S_F,k (x)}
S_{\mathbb{F}, k}(x) := \sum_{n=1}^{\infty}\mathtt{a}_{\mathbb{F}, k}(n)\tilde{Z}_{kr_1, kr_2}\left(\frac{2^{kr_2}\pi^{kd/2} n\sqrt{x}}{\sqrt{D^k}} \right), \quad |\Arg(x)| < \frac{\pi d}{2}.
\end{align}
Note that the function $\mathtt{a}_{\mathbb{F},k}(n)$ has the following bound, for any $\epsilon > 0$,
\begin{align}\label{bound for counting function}
\mathtt{a}_{\mathbb{F},k}(n) \ll n^{\epsilon},
\end{align}
which can be deduced using the bound for $\mathtt{a}_{\mathbb{F}}(n) \ll n^\epsilon$ due to Chandrasekharan and Narasimhan \cite[Lemma~9]{CN1963}. An upper bound for the function $\tilde{Z}_{r_1, r_2}(x)$ is given by \cite[Equation (2.3)]{DGV2022}
\begin{align}\label{bound for Z}
\tilde{Z}_{r_1, r_2}(y) \ll_{r_1, r_2}~ y^{-\frac{r_1 + r_2 -1}{d}}\exp{\left(-d\left(\frac{y}{2^{r_2}}\right)^{\frac{2}{d}}\right)},
\end{align}
as $y\rightarrow \infty$.
Using the bounds \eqref{bound for counting function} and \eqref{bound for Z}, one can easily see that the series $S_{\mathbb{F}, k}(x)$ converges absolutely. Moreover, as $x \rightarrow \infty$, we can check that
\begin{align}\label{Bound for S(x)}
S_{\mathbb{F}, k}(x) \ll_{k, r_1, r_2}~ x^{-\frac{kr_1 + kr_2 -1}{2kd}}\exp{\left(-\pi kd\left(\frac{x}{D^k}\right)^{\frac{1}{kd}}\right)}.
\end{align}
Using the integral representation \eqref{Defn of Z} of $\tilde{Z}_{r_1, r_2}(x)$, we rewrite $S_{\mathbb{F}, k}(x)$ as
\begin{align}\label{Theta reln 1}
S_{\mathbb{F}, k}(x) &= \sum_{n=1}^{\infty}\mathtt{a}_{\mathbb{F}, k}(n) \frac{1}{2\pi i}\int_{(c)}\Gamma^{kr_1}\left(\frac{s}{2}\right)\Gamma^{kr_2}\left(s\right)\left(\frac{2^{kr_2}\pi^{kd/2}n \sqrt{x}}{\sqrt{D^k}} \right)^{-s} \text{d}s.
\end{align}
Now we interchange the summation and integration by assuming $\Re(s) = c >1$ as the series $\sum_{n=1}^{\infty}\frac{\mathtt{a}_{\mathbb{F}, k}(n)}{n^{s}} = \zeta_{\mathbb{F}}^k(s)$ is uniformly and absolutely convergent in the region $\Re(s) > 1$. Hence, we have
\begin{align*}
S_{\mathbb{F}, k}(x) &= \frac{1}{2\pi i}\int_{(c)} \Gamma^{kr_1}\left(\frac{s}{2}\right)\Gamma^{kr_2}\left(s\right)\zeta_{\mathbb{F}}^k(s)\left(\frac{2^{kr_2}\pi^{kd/2} \sqrt{x}}{\sqrt{D^k}} \right)^{-s} \text{d}s \\
&= \frac{1}{2\pi i}\int_{(c)}\left(\frac{D}{4^{r_2}\pi^d}\right)^{\frac{ks}{2}}\Gamma^{kr_1}\left(\frac{s}{2}\right)\Gamma^{kr_2}(s)\zeta_{\mathbb{F}}^k(s) x^{-\frac{s}{2}} \text{d}s \\
&= \frac{1}{2\pi i}\int_{(c)}\Omega_{\mathbb{F}}^k(s)x^{-\frac{s}{2}} \text{d}s,
\end{align*}
where $\Omega_{\mathbb{F}}(s)$ is defined as in \eqref{Omega_F of fnal eqn}. We take a rectangular contour $\mathcal{C}$ made up of the vertices $c-iT, c+iT, \alpha+iT$ and $\alpha-iT$ taken in counter-clockwise direction. We already have $c > 1$ with some large value of $T$ and we choose $-1 < \alpha < 0$ so that $ \beta = 1-\alpha >1$. The integrand function $I(s):= \Omega_{\mathbb{F}}^k(s)x^{-s/2}$ has a pole of order $k$ at $s=0$ because of the pole of order $k(r_1 + r_2)$ by $\Gamma^{kr_1}\left(\frac{s}{2}\right)\Gamma^{kr_2}(s)$ and zero of order $k(r_1 + r_2 - 1)$ by $\zeta_\mathbb{F}^k(s)$ at $s=0$. Further, $I(s)$ has a pole of order $k$ at $s=1$ due to $\zeta_\mathbb{F}^k(s)$.
Therefore, by using Cauchy's residue theorem and taking $T \rightarrow \infty$, we have
\begin{align}
S_{\mathbb{F}, k}(x) &= R_0(x) + R_1(x) + \frac{1}{2\pi i}\int_{\alpha-i \infty}^{\alpha+i \infty}\Omega_{\mathbb{F}}^k(s)x^{-\frac{s}{2}} \text{d}s \nonumber \\
&= R_0(x) + R_1(x) + \frac{1}{2\pi i}\int_{\alpha-i \infty}^{\alpha+i \infty}\Omega_{\mathbb{F}}^k(1-s)x^{-\frac{s}{2}} \text{d}s \nonumber \\
&= R_0(x) + R_1(x) + \frac{1}{2\pi i}\int_{\beta-i \infty}^{\beta+i \infty}\Omega_{\mathbb{F}}^k(s)x^{-\frac{(1-s)}{2}} \text{d}s \nonumber \\
&= R_0(x) + R_1(x) + \frac{1}{\sqrt{x}}S_{\mathbb{F}, k}\left(\frac{1}{x} \right),\label{CRT Dedekind zeta}
\end{align}
where the terms $R_0(x)$ and $R_1(x)$ denote the residue of $I(s)$ at $s=0$ and $s=1$, respectively.
 The residues are given by
\begin{align}
R_0(x) &= \frac{1}{(k-1)!}\lim_{s \rightarrow 0}\frac{\text{d}^{k-1}}{\text{d}s^{k-1}}\left(s^k \Omega_{\mathbb{F}}^k(s)x^{-s/2}\right), \label{R_0(x)}\\
R_1(x) &= \frac{1}{(k-1)!}\lim_{s \rightarrow 1}\frac{\text{d}^{k-1}}{\text{d}s^{k-1}}\left((s-1)^k \Omega_{\mathbb{F}}^k(s)x^{-s/2}\right). \label{R_1(x)}
\end{align}
One can check that $R_1(x) = -\frac{1}{\sqrt{x}}R_0\left(\frac{1}{x}\right)$ with the help of the functional equation of $\zeta_{\mathbb{F}}^k(s)$ in $R_1(x)$. Employ this relation in \eqref{CRT Dedekind zeta} to get
\begin{align}\label{Transform for S_F,k}
S_{\mathbb{F}, k}(x) = R_0(x) - \frac{1}{\sqrt{x}}R_0\left(\frac{1}{x}\right) + \frac{1}{\sqrt{x}}S_{\mathbb{F}, k}\left(\frac{1}{x} \right).
\end{align}
Hence, we finally have
\begin{align}\label{End of thm 1 half}
W_{\mathbb{F}, k}\left(\frac{1}{x}\right) = \sqrt{x}W_{\mathbb{F}, k}(x),
\end{align}
where $W_{\mathbb{F}, k}(x) = S_{\mathbb{F}, k}(x) - R_0(x)$, which is exactly same as in \eqref{W_F, k(x)}.

Now we prove the converse of the theorem. We assume that \eqref{End of thm 1 half} holds. From \eqref{Defn of Z}, we know that $\tilde{Z}_{r_1, r_2}(t)$ is the inverse Mellin transform of $\Gamma^{r_1}\left(\frac{s}{2}\right)\Gamma^{r_2}(s)$. Therefore, we have
\begin{align*}
\Gamma^{kr_1}\left(\frac{s}{2}\right)\Gamma^{kr_2}(s) = \int_0^\infty t^{s-1}\tilde{Z}_{kr_1, kr_2}(t)\, \text{d}t, \quad \Re(s) > 0.
\end{align*}
Putting $t = \left(\frac{2^{kr_2}\pi^{kd/2} n\sqrt{x}}{\sqrt{D^k}} \right)$, one gets
\begin{align}\label{Change of variable}
\left(\frac{D}{4^{r_2}\pi^d}\right)^{\frac{ks}{2}}\Gamma^{kr_1}\left(\frac{s}{2}\right)\Gamma^{kr_2}(s)n^{-s} = \frac{1}{2}\int_0^\infty x^{\frac{s}{2}-1}\tilde{Z}_{kr_1, kr_2}\left(\frac{2^{kr_2}\pi^{kd/2} n\sqrt{x}}{\sqrt{D^k}} \right)\, \text{d}x.
\end{align}
We now assume $\Re(s) > 1$. Multiplying both sides of \eqref{Change of variable} by $\mathtt{a}_{\mathbb{F}, k}(n)$ and summing over $n$, and then using the definition \eqref{S_F,k (x)} of $S_{\mathbb{F}, k}(x)$, we get
\begin{align*}
\Omega_{\mathbb{F}}^k(s) &= \frac{1}{2}\int_0^\infty x^{\frac{s}{2}-1}\sum_{n=1}^{\infty}\mathtt{a}_{\mathbb{F}, k}(n)\tilde{Z}_{kr_1, kr_2}\left(\frac{2^{kr_2}\pi^{kd/2} n\sqrt{x}}{\sqrt{D^k}} \right)\, \text{d}x \\
&= \frac{1}{2}\int_0^\infty x^{\frac{s}{2}-1}S_{\mathbb{F}, k}(x) \, \text{d}x \\
&= \frac{1}{2}\int_0^1 x^{\frac{s}{2}-1}S_{\mathbb{F}, k}(x) \text{d}x + \frac{1}{2}\int_1^\infty x^{\frac{s}{2}-1}S_{\mathbb{F}, k}(x)\text{d}x \\
&= \frac{1}{2}\int_1^\infty x^{-\frac{s}{2}-1}S_{\mathbb{F}, k}\left(\frac{1}{x} \right) \text{d}x + \frac{1}{2}\int_1^\infty x^{\frac{s}{2}-1}S_{\mathbb{F}, k}(x) \text{d}x.
\end{align*}
Using \eqref{Transform for S_F,k} for the value of $S_{\mathbb{F}, k}\left(\frac{1}{x} \right)$, we have
\begin{align}
\Omega_{\mathbb{F}}^k(s) &= \frac{1}{2}\int_1^\infty x^{-\frac{s}{2}-1}\left(\sqrt{x}S_{\mathbb{F}, k}(x) - \sqrt{x}R_0(x) + R_0\left(\frac{1}{x}\right)\right) \text{d}x + \frac{1}{2}\int_1^\infty x^{\frac{s}{2}-1}S_{\mathbb{F}, k}(x)\text{d}x \nonumber \\
&= \frac{1}{2}\int_1^\infty x^{-\frac{s}{2}-1}\left(R_0\left(\frac{1}{x}\right) - \sqrt{x}R_0(x)\right) \text{d}x + \frac{1}{2}\int_1^\infty (x^{\frac{-s-1}{2}} + x^{\frac{s}{2}-1})S_{\mathbb{F}, k}(x) \text{d}x. \label{Symmetry of funl eqn general}
\end{align}
Using the bound \eqref{Bound for S(x)} for $S_{\mathbb{F}, k}(x)$, we can see that the second integral in \eqref{Symmetry of funl eqn general} converges absolutely for any $s \in \mathbb{C}$. Moreover, it is  symmetric over $s$ and $1-s$. Let us define the first integral in \eqref{Symmetry of funl eqn general} as 
$$
J(s) = \int_1^\infty x^{-\frac{s}{2}-1}\left(R_0\left(\frac{1}{x}\right) - \sqrt{x}R_0(x)\right) \text{d}x, \quad \Re(s) > 1.
$$
From \eqref{R_0(x)}, it is evident that $R_0(x)$ is some polynomial in $\mathbb{C}[\log(x)]$ of degree $k-1$.
Let us assume
\begin{align*}
R_0(x) &= a_0 + a_1 \log x + \cdots + a_{k-1}(\log x)^{k-1}, \\
\Rightarrow R_0\left(\frac{1}{x}\right) &= a_0 - a_1 \log x + \cdots + (-1)^{k-1}a_{k-1}(\log x)^{k-1},
\end{align*}
where $a_i's \in \mathbb{C}$. Putting these values of $R_0(x)$ and $R_0\left(\frac{1}{x}\right)$ in $J(s)$, we see that
\begin{align*}
J(s) &= \int_1^\infty a_0[x^{-\frac{s}{2}-1} - x^{\frac{-s-1}{2}}]\text{d}x + \int_1^\infty a_1\log x [-x^{-\frac{s}{2}-1} - x^{\frac{-s-1}{2}}]\text{d}x + \cdots \\
&+ \int_1^\infty a_{k-1}(\log x)^{k-1}[(-1)^{k-1}x^{-\frac{s}{2}-1} - x^{\frac{-s-1}{2}}]\text{d}x.
\end{align*}
Solving above integrals by substituting $x = e^u$ and integrating by parts, one gets
\begin{align*}
J(s) = 2a_0\left[\frac{1}{s} + \frac{1}{1-s}\right] - 4a_1\left[\frac{1}{s^2} + \frac{1}{(1-s)^2}\right] + \cdots + (-1)^{k-1}2^k a_{k-1}(k-1)!\left[\frac{1}{s^k} + \frac{1}{(1-s)^k}\right].
\end{align*}
The above expression of $J(s)$ suggests that it is analytic in the whole complex plane except at $s = 0$ and $s=1$ which are poles of order $k$.  Moreover, $J(s)$ is symmetric over $s$ and $1-s$. This gives the analytic continuation of $\zeta_\mathbb{F}^k(s)$ and symmetry of \eqref{Symmetry of funl eqn general} over $s$ and $1-s$, that is, $\Omega_{\mathbb{F}}^k(s) = \Omega_{\mathbb{F}}^k(1-s)$. This completes the proof.
\end{proof}

\begin{proof}[Corollary \rm{\ref{Jacobi theta for dedekind zeta}}][]
When $k=1$, $\mathtt{a}_{\mathbb{F}, k}(n) = \mathtt{a}_{\mathbb{F}}(n)$ and the residual term $R_0(x)$ in Theorem \ref{BM1} simplifies as
\begin{align}
R_0(x) &= \lim_{s \rightarrow 0}\left(s \Gamma^{r_1}\left(\frac{s}{2}\right)\Gamma^{r_2}(s)\zeta_\mathbb{F}(s)\left(\frac{2^{r_2}\pi^{d/2} \sqrt{x}}{\sqrt{D}} \right)^{-s}\right) \nonumber \\
&= \lim_{s \rightarrow 0} \left(s^{r_1+r_2}\Gamma^{r_1}\left(\frac{s}{2}\right)\Gamma^{r_2}(s)\frac{\zeta_{\mathbb{F}}(s)}{s^{r_1+r_2-1}}\left(\frac{2^{2r_2}\pi^d x}{D}\right)^{-\frac{s}{2}}\right) \nonumber \\
&= 2^{r_1}C_{\mathbb{F}}. \label{Residue}
\end{align}
In the last step, we have used \eqref{Laurent series_at s=0_1st coeff}.
Putting the above values in \eqref{Theta reln for dedekind zeta power k}, we get the desired result \eqref{Theta reln for dedekind zeta}.
\end{proof}

\begin{proof}[Corollary \rm{\ref{Theta relation for real zeta power k}}][]
In this case, we have $r_2 = 0, d = r_1$ since $\mathbb{F}$ is totally real. From \eqref{Z value real}, we know that $\tilde{Z}_{r_1, 0}(x) = 2V(x^2|{\bf \bar{0}}_{r_1})$. Hence, the term in \eqref{W_F, k(x)} becomes
\begin{align*}
\tilde{Z}_{kr_1, kr_2}\left(\frac{2^{kr_2}\pi^{kd/2} n\sqrt{x}}{\sqrt{D^k}} \right) = 2V\left(\frac{\pi^{kr_1} n^2 x}{D^k} \bigg| \bar{0}_{kr_1} \right),
\end{align*}
and
\begin{align*}
R_0(x)
&= \frac{1}{(k-1)!}\lim_{s \rightarrow 0}\frac{\mathrm{d}^{k-1}}{\mathrm{d}s^{k-1}}\left(s^k \Gamma^{kr_1}\left(\frac{s}{2}\right)\zeta_\mathbb{F}^k(s)\left(\frac{\pi^{(kr_{1}/2)} \sqrt{x}}{\sqrt{D^k}} \right)^{-s}\right).
\end{align*}
We use the above values in Theorem \ref{BM1} to get the desired result.
\end{proof}

%

\begin{proof}[Corollary \rm{\ref{Theta relation for real quadratic field}}][]
In this case, we have $r_1 = 2$ and $k = 1$. From \eqref{Bessel fn formula}, we have $V(z| 0, 0) = 2\text{K}_{0}(2\sqrt{z})$. Hence the term $W_{\mathbb{F}, 1}(x)$ in Corollary \ref{Theta relation for real zeta power k} reduces to 
\begin{align*}
W_{\mathbb{F}, 1}(x) 
&= 2\sum_{n=1}^{\infty}\mathtt{a}_{\mathbb{F}}(n)V\left(\frac{\pi^{2} n^2 x}{D} \bigg| 0, 0 \right) - R_0(x) \\
&= 4\sum_{n=1}^{\infty}\mathtt{a}_{\mathbb{F}}(n)K_0\left(\frac{2\pi n \sqrt{x}}{\sqrt{D}}\right) - 4C_{\mathbb{F}}.
\end{align*}
Here we have used \eqref{Residue} in the final step.
Hence substituting the above value in Corollary \ref{Theta relation for real zeta power k}, we prove the result.
\end{proof}

\begin{proof}[Corollary \rm{\ref{Equivalence for zeta power k}}][]
The proof of this corollary immediately follows by taking $\mathbb{F} = \mathbb{Q}$ in Corollary \ref{Theta relation for real zeta power k}.
\end{proof}

\begin{proof}[Theorem \rm{\ref{BM3}}][]
Let us define
\begin{align}\label{Defn of L}
L_{\mathbb{F}, -k}(x) := \sum_{n=1}^{\infty}\frac{\mu_{\mathbb{F}, k}(n)}{n}Z_{kr_1, kr_2}\left(\frac{2^{kr_2}\pi^{kd/2} \sqrt{x}}{n\sqrt{D^k}} \right).
\end{align}
From the definition \eqref{Defn of Z without tilde} of ${Z}_{r_1, r_2}\left(x\right)$, we have
\begin{align*}
L_{\mathbb{F}, -k}(x) 
&= \sum_{n=1}^{\infty}\frac{\mu_{\mathbb{F}, k}(n)}{n} \frac{1}{2\pi i}\int_{(b)}\Gamma^{kr_1}\left(\frac{s}{2}\right)\Gamma^{kr_2}(s)\left(\frac{2^{kr_2}\pi^{kd/2} \sqrt{x}}{n\sqrt{D^k}} \right)^{-s} \text{d}s,
\end{align*}
where $-1<b<0$.
Interchanging the summation and integration, as the series $\sum_{n=1}^{\infty}\mu_{\mathbb{F}, k}(n)n^{s-1} = \zeta_{\mathbb{F}}^{-k}(1-s)$ is uniformly and absolutely convergent for $-1< \Re(s) = b < 0$, so we have
\begin{align*}
L_{\mathbb{F}, -k}(x) &= \frac{1}{2\pi i}\int_{(b)}\left(\frac{D}{4^{r_2}\pi^d}\right)^{\frac{ks}{2}}\Gamma^{kr_1}\left(\frac{s}{2}\right)\Gamma^{kr_2}(s) \zeta_{\mathbb{F}}^{-k}(1-s) x^{-\frac{s}{2}} \text{d}s  \\
&= \frac{1}{2\pi i}\int_{(b)}\Lambda_{\mathbb{F}}^{k}(s)x^{-\frac{s}{2}} \text{d}s,
\end{align*}
where $\Lambda_{\mathbb{F}}^{k}(s)$ is defined as in \eqref{Functional eqn for dedekind zeta power -k}.
Consider a rectangular contour $\mathcal{C}$ with vertices $\alpha-iT$, $\alpha+iT, b+iT$, and $b-iT$ taken in counter-clockwise direction. We have $-1<b<0$ with large value of $T$ and we choose $1 < \alpha < 2$ so that $ \beta = 1-\alpha <0$. The integrand function has a pole of order $kr$  at $s=0$, where $r = r_1 + r_2 -1$, arising from a pole of order $k(r_1 + r_2)$ due to $\Gamma^{kr_1}\left(\frac{s}{2}\right)\Gamma^{kr_2}(s)$ and a zero of order $k$ from $\zeta_\mathbb{F}^{-k}(1-s)$ at $s=0$. Additionally, the integrand has a pole of order $kr$ at $s=1$, due to $\zeta_{\mathbb{F}}^{-k}(1-s)$. As $T \rightarrow \infty$, the integrand encounters infinitely many poles of order $k$ at the non-trivial zeros $s = \rho$ of $\zeta_\mathbb{F}(s)$, that lie in the critical strip $0 < \Re(s) < 1$. Here, we are assuming the simplicity of the non-trivial zeros of $\zeta_\mathbb{F}(s)$.

Hence, using Cauchy's residue theorem and taking $T \rightarrow \infty$,
we have
{\allowdisplaybreaks \begin{align}
L_{\mathbb{F}, -k}(x) &=  \frac{1}{2\pi i}\int_{\alpha-i \infty}^{\alpha+i \infty}\Lambda_{\mathbb{F}}^{k}(s)x^{-\frac{s}{2}} \text{d}s - R_0(x) - R_1(x) - \sum_{\rho} R_{\rho}(x) \nonumber \\
&= \frac{1}{2\pi i}\int_{\alpha-i \infty}^{\alpha+i \infty}\Lambda_{\mathbb{F}}^{k}(1-s)x^{-\frac{s}{2}} \text{d}s - R_0(x) - R_1(x) - \sum_{\rho} R_{\rho}(x) \nonumber \\
&= \frac{1}{2\pi i}\int_{\beta-i \infty}^{\beta+i \infty}\Lambda_{\mathbb{F}}^{k}(s)x^{-\frac{(1-s)}{2}} \text{d}s - R_0(x) - R_1(x) - \sum_{\rho} R_{\rho}(x) \nonumber \\
&= \frac{1}{\sqrt{x}}L_{\mathbb{F}, -k}\left(\frac{1}{x} \right) - R_0(x) - R_1(x) - \sum_{\rho} R_{\rho}(x),\label{CRT Dedekind zeta for negative k}
\end{align}}
where the terms $R_0(x), R_1(x)$ and $R_\rho(x)$ denote the residue of $\Lambda_{\mathbb{F}}^{k}(s)x^{-\frac{s}{2}}$ at $s=0, s=1$ and $s=\rho$ respectively. The residues are given by
{\allowdisplaybreaks
\begin{align}
R_0(x) &= \frac{1}{(kr-1)!}\lim_{s \rightarrow 0}\frac{\text{d}^{kr-1}}{\text{d}s^{kr-1}}\left(s^{kr} \Lambda_{\mathbb{F}}^{k}(s)( \sqrt{x})^{-s}\right), \label{Def R_0(x)}\\
R_1(x) &= \frac{1}{(kr-1)!}\lim_{s \rightarrow 1}\frac{\text{d}^{kr-1}}{\text{d}s^{kr-1}}\left((s - 1)^{kr} \Lambda_{\mathbb{F}}^{k}(s)( \sqrt{x})^{-s}\right), \label{Def R_1(x)}\\
R_\rho(x) &= \frac{1}{(k-1)!}\lim_{s \rightarrow \rho}\frac{\text{d}^{k-1}}{\text{d}s^{k-1}}\left((s-\rho)^k \Lambda_{\mathbb{F}}^{k}(s)(\sqrt{x})^{-s}\right). \label{Def R_rho(x)}
\end{align}}
It can be verified that $R_1(x) = -\frac{1}{\sqrt{x}}R_0\left(\frac{1}{x}\right)$ by applying functional equation of $\zeta_{\mathbb{F}}^{-k}(s)$ in $R_1(x)$. Utilizing this relation in \eqref{CRT Dedekind zeta for negative k}, we get
\begin{align*}
L_{\mathbb{F}, -k}(x) + R_0(x) = \frac{1}{\sqrt{x}}\left[L_{\mathbb{F}, -k}\left(\frac{1}{x} \right) + R_0\left(\frac{1}{x}\right)\right] - \sum_{\rho} R_{\rho}(x).
\end{align*}
Now letting $\alpha\beta=1$ in the last equation, where $\alpha = x$ and $\beta = \frac{1}{x}$, we obtain
\begin{align}\label{Symmetry in alpha beta}
\alpha^{\frac{1}{4}}\left[L_{\mathbb{F}, -k}\left(\alpha\right) + R_0(\alpha)\right] &= \beta^{\frac{1}{4}}\left[L_{\mathbb{F}, -k}\left(\beta\right) + R_0(\beta) \right]  - \alpha^{\frac{1}{4}}\sum_{\rho} R_{\rho}(\alpha).
\end{align}
Since $\alpha\beta=1$, we replace $\alpha$ by $\beta$ in \eqref{Symmetry in alpha beta} to get
\begin{align}\label{Symmetry in beta alpha}
\beta^{\frac{1}{4}}\left[L_{\mathbb{F}, -k}\left(\beta\right) + R_0(\beta) \right] &=  \alpha^{\frac{1}{4}}\left[L_{\mathbb{F}, -k}\left(\alpha\right) + R_0(\alpha)\right] -  \beta^{\frac{1}{4}}\sum_{\rho} R_{\rho}(\beta).
\end{align}
Adding \eqref{Symmetry in alpha beta} and \eqref{Symmetry in beta alpha}, we obtain
\begin{align}\label{alpha beta residue equal}
\alpha^{\frac{1}{4}}\sum_{\rho} R_{\rho}(\alpha) &= - \beta^{\frac{1}{4}}\sum_{\rho} R_{\rho}(\beta).
\end{align}
We now substitute the relation \eqref{alpha beta residue equal} into \eqref{Symmetry in alpha beta}, which yields
\begin{align*}
\alpha^{\frac{1}{4}}\left[L_{\mathbb{F}, -k}\left(\alpha\right) + R_0(\alpha) + \frac{1}{2}\sum_{\rho} R_{\rho}(\alpha)\right] &= \beta^{\frac{1}{4}}\left[L_{\mathbb{F}, -k}\left(\beta\right) + R_0(\beta) + \frac{1}{2}\sum_{\rho} R_{\rho}(\beta)\right].
\end{align*}
Putting back the value of $\alpha = x$ and $\beta = \frac{1}{x}$ in the above equation, we get
\begin{align}\label{Another form of L trans}
L_{\mathbb{F}, -k}\left(x\right) + R_0(x) + \frac{1}{2}\sum_{\rho} R_{\rho}(x) = \frac{1}{\sqrt{x}}\left[L_{\mathbb{F}, -k}\left(\frac{1}{x}\right) + R_0\left(\frac{1}{x}\right) + \frac{1}{2}\sum_{\rho} R_{\rho}\left(\frac{1}{x}\right)\right].
\end{align}
Thus, we have
\begin{align*}
U_{\mathbb{F}, -k}\left(\frac{1}{x}\right) = \sqrt{x}U_{\mathbb{F}, -k}(x),
\end{align*}
where $U_{\mathbb{F}, -k}(x) = L_{\mathbb{F}, -k}\left(x\right) + R_0(x) + \frac{1}{2}\sum_{\rho} R_{\rho}(x)$, which is same as in \eqref{U_F, k(x)}. This completes the proof of \eqref{Theta reln for dedekind zeta power -k}.

We now proceed to prove the functional equation for $\zeta_{\mathbb{F}}^{-k}(s)$ by assuming theta relation \eqref{Theta reln for dedekind zeta power -k}. It follows from \eqref{Defn of Z without tilde}, for $-1< \Re(s) <0$, that
\begin{align*}
\Gamma^{kr_1}\left(\frac{s}{2}\right)\Gamma^{kr_2}(s) = \int_0^\infty t^{s-1}{Z}_{kr_1, kr_2}(t)\, \text{d}t.
\end{align*}
Putting $t = \left(\frac{2^{kr_2}\pi^{kd/2}\sqrt{x}}{n\sqrt{D^k}} \right)$, one gets
\begin{align}\label{Change of variable for negative k}
\left(\frac{D}{4^{r_2}\pi^d}\right)^{\frac{ks}{2}}\Gamma^{kr_1}\left(\frac{s}{2}\right)\Gamma^{kr_2}(s)n^{s} = \frac{1}{2}\int_0^\infty x^{\frac{s}{2}-1}{Z}_{kr_1, kr_2}\left(\frac{2^{kr_2}\pi^{kd/2} \sqrt{x}}{n\sqrt{D^k}} \right)\, \text{d}x.
\end{align}
Multiplying both sides of \eqref{Change of variable for negative k} by $\frac{\mu_{\mathbb{F}, k}(n)}{n}$, summing over $n$, and referring to the definition \eqref{Defn of L} of $L_{\mathbb{F}, -k}\left(x\right)$, we get
\begin{align*}
\Lambda_{\mathbb{F}}^{k}(s) &= \frac{1}{2}\int_0^\infty x^{\frac{s}{2}-1}\sum_{n=1}^{\infty}\frac{\mu_{\mathbb{F}, k}(n)}{n}{Z}_{kr_1, kr_2}\left(\frac{2^{kr_2}\pi^{kd/2} \sqrt{x}}{n\sqrt{D^k}} \right)\, \text{d}x \\
&= \frac{1}{2}\int_0^\infty x^{\frac{s}{2}-1}L_{\mathbb{F}, -k}(x) \, \text{d}x \\
&= \frac{1}{2}\int_0^1 x^{\frac{s}{2}-1}L_{\mathbb{F}, -k}(x) \text{d}x + \frac{1}{2}\int_1^\infty x^{\frac{s}{2}-1}L_{\mathbb{F}, -k}(x)\text{d}x \\
&= \frac{1}{2}\int_0^1 x^{\frac{s}{2}-1}L_{\mathbb{F}, -k}(x) \text{d}x + \frac{1}{2}\int_0^1 x^{-\frac{s}{2}-1}L_{\mathbb{F}, -k}\left(\frac{1}{x}\right) \text{d}x.
\end{align*}
Using \eqref{Another form of L trans}, which is same as \eqref{Theta reln for dedekind zeta power -k}, for the value of $L_{\mathbb{F}, -k}\left(\frac{1}{x} \right)$, we have
\begin{align}
\Lambda_{\mathbb{F}}^{k}(s) &= \frac{1}{2}\int_0^1 x^{-\frac{s}{2}-1}\left[\sqrt{x}L_{\mathbb{F}, -k}(x) + \sqrt{x}R_0(x) + \frac{\sqrt{x}}{2}\sum_{\rho} R_{\rho}(x) - R_0\left(\frac{1}{x}\right) - \frac{1}{2}\sum_{\rho} R_{\rho}\left(\frac{1}{x}\right)\right] \text{d}x \nonumber\\
&+ \frac{1}{2}\int_0^1 x^{\frac{s}{2}-1}L_{\mathbb{F}, -k}(x)\text{d}x, \nonumber \\
&= \frac{1}{2}\int_0^1 \frac{x^{-\frac{s}{2}}}{x}\left[\sqrt{x}R_0(x) - R_0\left(\frac{1}{x}\right) + \frac{\sqrt{x}}{2}\sum_{\rho} R_{\rho}(x) - \frac{1}{2}\sum_{\rho} R_{\rho}\left(\frac{1}{x}\right)\right] \text{d}x  \nonumber\\ 
& + \frac{1}{2}\int_0^1 (x^{\frac{-s-1}{2}} + x^{\frac{s}{2}-1})L_{\mathbb{F}, -k}(x) \text{d}x. \label{Symmetry of funl eqn general for negative k}
\end{align}
The second integral in the above expression \eqref{Symmetry of funl eqn general for negative k} is symmetric over $s$ and $1-s$. We denote the first integral in \eqref{Symmetry of funl eqn general for negative k} by 
\begin{align*}
K(s) &:= \int_0^1 x^{-\frac{s}{2}-1}\left[\sqrt{x}R_0(x) - R_0\left(\frac{1}{x}\right) + \frac{\sqrt{x}}{2}\sum_{\rho} R_{\rho}(x) - \frac{1}{2}\sum_{\rho} R_{\rho}\left(\frac{1}{x}\right) \right] \text{d}x \\
&= \int_0^1 \left[x^{\frac{-s-1}{2}}R_0(x) - x^{-\frac{s}{2}-1}R_0\left(\frac{1}{x}\right)\right] \text{d}x + \frac{1}{2}\sum_{\rho}\int_0^1 \left[x^{\frac{-s-1}{2}}R_{\rho}(x) - x^{-\frac{s}{2}-1}R_{\rho}\left(\frac{1}{x}\right)\right] \text{d}x. 
\end{align*}
From \eqref{Def R_0(x)}, one can observe that $R_0(x)$ is a polynomial of degree $kr-1$ in $\mathbb{C}[\log x]$. So, let us suppose that 
\begin{align*}
R_0(x) &= a_0 + a_1 \log x + \cdots + a_{kr-1}(\log x)^{kr-1} \\
\implies R_0\left(\frac{1}{x}\right) &= a_0 - a_1 \log x + \cdots + (-1)^{kr-1}a_{kr-1}(\log x)^{kr-1}.
\end{align*}
Furthermore, from \eqref{Def R_rho(x)}, $R_{\rho}(x)$ is also a polynomial of the form $x^{-\frac{\rho}{2}}g(x)$, where $g(x)$ is some polynomial of degree $k-1$ in $\mathbb{C}[\log x]$. So, we write $R_{\rho}(x)$ as
\begin{align}
R_{\rho}(x) &= \frac{1}{x^{\frac{\rho}{2}}}\left(b_0 + b_1 \log x + \cdots + b_{k-1}(\log x)^{k-1}\right) \nonumber \\
\implies R_{\rho}\left(\frac{1}{x}\right) &= x^{\frac{\rho}{2}}\left(b_0 - b_1 \log x + \cdots + (-1)^{k-1}b_{k-1}(\log x)^{k-1}\right). \nonumber
\end{align}
Hence, using the above values of $R_0(x), R_0\left( \frac{1}{x} \right), R_\rho(x)$ and $R_\rho\left( \frac{1}{x} \right)$ in $K(s)$, we see that
{\allowdisplaybreaks
\begin{align*}
K(s) &= \int_0^1 a_0\left[x^{\frac{-s-1}{2}} - x^{-\frac{s}{2}-1}\right]\text{d}x + \int_0^1 a_1\log x \left[x^{\frac{-s-1}{2}} + x^{-\frac{s}{2}-1}\right]\text{d}x + \cdots \\
&+ \int_0^1 a_{kr-1}(\log x)^{kr-1}\left[x^{\frac{-s-1}{2}} + (-1)^{kr}x^{-\frac{s}{2}-1}\right]\text{d}x \\
&+ \frac{1}{2}\sum_{\rho}\bigg(\int_0^1 b_0\left[x^{\frac{-s-1-\rho}{2}} - x^{\frac{\rho - s}{2}-1}\right]\text{d}x + \int_0^1 b_1\log x \left[x^{\frac{-s-1-\rho}{2}} + x^{\frac{\rho - s}{2}-1}\right]\text{d}x + \cdots \\
&+ \int_0^1 b_{k-1}(\log x)^{k-1}\left[x^{\frac{-s-1-\rho}{2}} + (-1)^{k}x^{\frac{\rho - s}{2}-1} \right]\text{d}x \bigg).
\end{align*}}
One can check that the above integrals converge for $\Re(s)<0$ and  $\Re(s+\rho)<1$, which is true as $\rho$ lies in the critical strip.   
Solving the above integrals by substituting $x = e^u$ and integrating by parts, we get

\begin{align*}
K(s) &= 2a_0\left[\frac{1}{s} + \frac{1}{1-s}\right] - 4a_1\left[\frac{1}{s^2} + \frac{1}{(1-s)^2}\right] + \cdots \\
&- (-2)^{kr}a_{kr-1}(kr-1)!\left[\frac{1}{s^{kr}} + \frac{1}{(1-s)^{kr}}\right] + \frac{1}{2}\sum_{\rho} \bigg(2b_0\left[\frac{1}{s-\rho} + \frac{1}{1-s-\rho}\right] \\
&- 4b_1\left[\frac{1}{(s-\rho)^2} + \frac{1}{(1-s-\rho)^2}\right] + \cdots - (-2)^{k}b_{k-1}(k-1)!\left[\frac{1}{(s-\rho)^{k}} + \frac{1}{(1-s-\rho)^{k}}\right]\bigg).
\end{align*}
Therefore, from the above expression it is evident that $K(s)$ is symmetric over $s$ and $1-s$. This gives the symmetry of \eqref{Symmetry of funl eqn general for negative k} over $s$ and $1-s$. This completes the proof of Theorem \ref{BM3}.  
\end{proof}

\begin{proof}[Corollary \rm{\ref{Theta reln for zeta power -k}}][]
When $\mathbb{F} = \mathbb{Q}$, $\zeta_\mathbb{F}(s) = \zeta(s)$. From \eqref{Mobius for Q}, we have $\mu_{\mathbb{Q}, k}(n) = \mu_{k}(n)$. Using \eqref{Z without tilde at real} for the value of $Z_{k, 0}\left(\frac{\pi^{k/2} \sqrt{x}}{n}\right)$ in Theorem \ref{BM3} and observing that the residual term $R_0(x)$ vanishes as we are dealing with $\mathbb{F} = \mathbb{Q}$, this completes the proof. 
\end{proof}

\begin{proof}[Corollary \rm{\ref{Theta reln for zeta power -2}}][]
When $k=2$, from \eqref{Bessel fn formula}, we have 
$$
V\left(\frac{\pi^{k} x}{n^2} \bigg | \bar{0}_{k} \right) = 2K_0\left(\frac{2\pi\sqrt{x}}{n}\right),
$$
where $K_0(x)$ is the modified Bessel function of second kind. The residual term $\mathfrak{R}_0(x)$ in Corollary \ref{Theta reln for zeta power -k} is given by
\begin{align*}
\mathfrak{R}_0\left(\frac{\pi \sqrt{x}}{n}\right) &= \lim_{s \rightarrow 0}\frac{{\rm d}}{{\rm d}s}\left(s^2 \Gamma^2\left(\frac{s}{2}\right)\left(\frac{\pi\sqrt{x}}{n}\right)^{-s}\right).\\
\end{align*}
To solve the above limit, we use the Laurent series expansions given below
\begin{align*}
s^2\Gamma^2\left(\frac{s}{2}\right) &= 4 - 4\gamma s + \gamma^2 s^2 + \cdots, \\
\left(\frac{\pi\sqrt{x}}{n}\right)^{-s} &= 1 - s\log\left(\frac{\pi\sqrt{x}}{n}\right) + \cdots.
\end{align*}
Hence, $\mathfrak{R}_0\left(\frac{\pi \sqrt{x}}{n}\right) = -4\gamma - 4\log\left(\frac{\pi \sqrt{x}}{n}\right)$.
So, we have
\begin{align*}
U_{\mathbb{Q}, -k}(x) &= 4\sum_{n=1}^{\infty}\frac{\mu_{2}(n)}{n}\left[K_0\left(\frac{2\pi\sqrt{x}}{n}\right) + \gamma + \log\left(\frac{\pi \sqrt{x}}{n}\right) \right] + \frac{1}{2}\sum_{\rho} R_{\rho}(x) \\
&= 4\sum_{n=1}^{\infty}\frac{\mu_{2}(n)}{n}\left[K_0\left(\frac{2\pi\sqrt{x}}{n}\right) + \log\left(\frac{\pi \sqrt{x}}{n}\right) \right] + \frac{1}{2}\sum_{\rho} R_{\rho}(x),
\end{align*}
where in the final step, we have used the fact that $\sum_{n=1}^{\infty}\frac{\mu_{2}(n)}{n} = 0$, see \eqref{Sum of 1/Dzeta(1)}. 
Hence using all the above values in Corollary \ref{Theta reln for zeta power -k}, we have the desired result.
\end{proof}

\begin{proof}[Theorem \rm{\ref{BM2}}][]
When $\mathbb{F} = \mathbb{Q}$, the Dedekind zeta function is nothing but the Riemann zeta function and existence of infinitely many non-trivial zeros for $\zeta(s)$ on the critical line has been proved by Hardy \cite{Hardy}. For the case of any quadratic field $\mathbb{F}$, the Dedekind zeta function can be written as $\zeta_\mathbb{F}(s) = \zeta(s)L(s, \chi)$ for some Dirichlet character $\chi$. Since $L(s, \chi)$ is entire, so $\zeta_\mathbb{F}(s)$ has infinitely many non-trivial zeros on the critical line due to $\zeta(s)$. In this proof, we consider the number fields $\mathbb{F}$ of degree $d \geq 3$.
First, let us recall the function 
\begin{align}
\Xi_{\mathbb{F}}(t)= \left(-\frac{1}{8} - \frac{t^2}{2} \right)\left(\frac{D}{4^{r_2}\pi^d}\right)^{\frac{1}{4}+\frac{it}{2}}\Gamma^{r_1}\left(\frac{1}{4} + \frac{it}{2}\right)\Gamma^{r_2}\left(\frac{1}{2} + it\right)\zeta_{\mathbb{F}}\left(\frac{1}{2} + it\right).
\end{align}
which we studied in Section \ref{Key Tools}.
One can observe that $\Xi_{\mathbb{F}}(t)$ vanishes only when $\zeta_\mathbb{F}\left(\frac{1}{2} + it\right)$ is zero for some $t \in \mathbb{R}$ since all the other factors never vanish for any real $t$. This implies that each real zero of $\Xi_{\mathbb{F}}(t)$ corresponds to a zero of $\zeta_\mathbb{F}(s)$ on the half line. So, our main aim is to prove that $\Xi_{\mathbb{F}}(t)$ has infinitely many real zeros. From Lemma \ref{Xi as even and real fn}, we know that the function $\Xi_{\mathbb{F}}(t)$ is even and real for any $t \in \mathbb{R}$.

Now, we define a function $f(t)$ such that $f(t) := |\phi(it)|^2 = \phi(it)\phi(-it), \,\, t \in \mathbb{R}$, where $\phi(s)$ is an analytic function with the characterisation $\overline{\phi(s)} = \phi(\bar{s})$.
Let us examine the following integral, for $z \in \mathbb{C}$,
\begin{align}\label{Defn of Phi}
\Phi(z) = \int_0^\infty f(t)\Xi_{\mathbb{F}}(t)\cos(zt) \text{d}t.
\end{align}
We write $\cos(zt) = \frac{e^{izt} + e^{-izt}}{2}$ and substitute $e^z=y$ to see
\begin{align*}
\Phi(z) &= \frac{1}{2}\int_0^\infty f(t)\Xi_{\mathbb{F}}(t)(y^{it} + y^{-it})\text{d}t \\
&= \frac{1}{2}\int_0^\infty f(t)\Xi_{\mathbb{F}}(t)y^{it}\text{d}t + \frac{1}{2}\int_0^\infty f(t)\Xi_{\mathbb{F}}(t)y^{-it}\text{d}t.
\end{align*}
Now we change the variable $t$ by $-t$ in the second integral and use the fact that $f(t)$ and $\Xi_{\mathbb{F}}(t)$ are real and even functions of $t$. Then, we have
\begin{align*}
\Phi(z) &= \frac{1}{2}\int_0^\infty f(t)\Xi_{\mathbb{F}}(t)y^{it}\text{d}t + \frac{1}{2}\int_{-\infty}^0 f(t)\Xi_{\mathbb{F}}(t)y^{it}\text{d}t \\
&= \frac{1}{2}\int_{-\infty}^\infty f(t)\Xi_{\mathbb{F}}(t)y^{it}\text{d}t \\
&= \frac{1}{2}\int_{-\infty}^\infty \phi(it)\phi(-it)\xi_{\mathbb{F}}\left(\frac{1}{2} + it\right) y^{it}\text{d}t.
\end{align*}
Letting $\frac{1}{2}+it = s$ and $\phi(s) = \frac{1}{s+\frac{1}{2}}$ in the above integral and using \eqref{xi function}, one can see that
\begin{align}
\Phi(z) &= \frac{1}{2i\sqrt{y}}\int_{\frac{1}{2} - i\infty}^{\frac{1}{2} + i\infty}\frac{1}{s(1-s)}\frac{s(s-1)}{2}\Omega_{\mathbb{F}}(s)y^s \text{d}s \nonumber \\
&= -\frac{1}{2\pi i}\times\frac{\pi}{2\sqrt{y}}\int_{\frac{1}{2} - i\infty}^{\frac{1}{2} + i\infty}\Omega_{\mathbb{F}}(s)y^s \text{d}s \nonumber \\
&= -\frac{\pi}{2\sqrt{y}}\left(W_{\mathbb{F},1}\left(\frac{1}{y^2}\right) + C_\mathbb{F}2^{r_1}(1+y)\right). \label{Phi after Lemma 3.2}
\end{align}
Here, we have used Lemma \ref{Lemma 1} in the last step. Since $y = e^z$ and $f(t) = \phi(it)\phi(-it) = \frac{1}{t^2 + \frac{1}{4}}$, so from \eqref{Defn of Phi} and \eqref{Phi after Lemma 3.2}, we have
\begin{align}\label{Final Phi with cos}
\Phi(z) = \int_0^\infty \frac{\Xi_{\mathbb{F}}(t)}{t^2 + \frac{1}{4}}\cos(zt)\text{d}t = -\frac{\pi}{2}\left[e^{-\frac{z}{2}}W_{\mathbb{F},1}(e^{-2z}) + C_\mathbb{F}2^{r_1}\left(e^{-\frac{z}{2}} + e^{\frac{z}{2}}\right)\right].
\end{align}
Now letting $z=-i\alpha$ in \eqref{Final Phi with cos} and using the fact that $\cos(-it) = \cosh(t)$, we get
\begin{align}
\int_0^\infty \frac{\Xi_{\mathbb{F}}(t)}{t^2 + \frac{1}{4}}\cosh(\alpha t)\text{d}t &= -\frac{\pi}{2}\left[C_\mathbb{F}2^{r_1}\left(e^{\frac{-i\alpha}{2}} + e^{\frac{i\alpha}{2}}\right) + e^{\frac{i\alpha}{2}}W_{\mathbb{F},1}\left(e^{2i\alpha}\right)\right] \nonumber \\
&= -\frac{\pi}{2}\left[2C_\mathbb{F}2^{r_1}\left(\frac{e^{\frac{i\alpha}{2}} + e^{-\frac{i\alpha}{2}}}{2}\right) + e^{\frac{i\alpha}{2}}W_{\mathbb{F},1}(e^{2i\alpha})\right] \nonumber \\
&= -\pi C_\mathbb{F}2^{r_1}\cos\left(\frac{\alpha}{2}\right) - \frac{\pi e^{\frac{i\alpha}{2}}}{2}W_{\mathbb{F},1}(e^{2i\alpha}). \label{Phi with cosh}
\end{align}
Note that the above identity is valid for $|\alpha| < \frac{\pi d}{4}$ since the factor $W_{\mathbb{F},1}(x)$ converges for $|\Arg(x)| < \frac{\pi d}{2}$, see \eqref{Theta reln for dedekind zeta}. Next, we need to show that the equation \eqref{Phi with cosh} can be differentiated with respect to $\alpha$ any number of times. One can easily prove that
$$
\Xi_{\mathbb{F}}(t) = O\left(|t|^A D^{\frac{1}{4} + \epsilon} e^{-\frac{\pi dt}{4}} \right)
$$ 
for some positive $A$, by using the bounds $\zeta_{\mathbb{F}}\left(\frac{1}{2} + it \right) = O(|t|^{\frac{d}{4} + \epsilon} D^{\frac{1}{4} + \epsilon})$ and $|\Gamma(\frac{1}{2} + i t)| = O(e^{-\frac{\pi}{2}|t|})$ as $|t| \rightarrow \infty$.  Also, it is well known that the function $\cosh(x)$ is bounded above by $e^{|x|}$. Using the above bounds for $\Xi_{\mathbb{F}}(t)$ and $\cosh(x)$, we see that the improper integral in \eqref{Phi with cosh} is absolutely convergent for $|\alpha| < \frac{\pi d}{4}$. Thus, we differentiate \eqref{Phi with cosh} $2n$-times with respect to $\alpha$, to obtain
\begin{align}\label{differentiated 2n times}
\int_0^\infty \frac{\Xi_{\mathbb{F}}(t)}{\left(t^2 + \frac{1}{4}\right)}t^{2n}\cosh(\alpha t)\text{d}t &= (-1)^{n+1}\frac{\pi C_\mathbb{F}2^{r_1}\cos\left(\frac{\alpha}{2}\right)}{2^{2n}} - \frac{\pi}{2} \frac{\text{d}^{2n}}{\text{d}\alpha^{2n}}\left(e^{\frac{i\alpha}{2}} W_{\mathbb{F},1}(e^{2i\alpha})\right).
\end{align}
Our next claim is that for any fixed $n$, the second term on the right hand side of the above equation tends to zero as $\alpha \rightarrow \frac{\pi}{2}$. 
From Corollary \ref{Jacobi theta for dedekind zeta}, we have
\begin{align*}
W_{\mathbb{F}, 1}\left(e^{-2i\alpha}\right) = e^{i\alpha}W_{\mathbb{F}, 1}(e^{2i\alpha}), \quad \text{for} \,\,\, |\alpha| < \frac{\pi d}{4}.
\end{align*}
As we are dealing with number fields of degree $d \geq 3$, so from the above equation, one can easily see that $W_{\mathbb{F},1}(e^{2i\alpha})$ and all its derivatives tend to zero as $\alpha \rightarrow \frac{\pi}{2}$ along any path. Therefore, from \eqref{differentiated 2n times}, we finally have
\begin{align}\label{Integral with limit}
\lim_{\alpha \rightarrow \frac{\pi}{2}} \int_0^\infty \frac{\Xi_{\mathbb{F}}(t)}{\left(t^2 + \frac{1}{4}\right)}t^{2n}\cosh(\alpha t)\text{d}t &= (-1)^{n+1}\frac{C_\mathbb{F}2^{r_1}\pi\cos\left(\frac{\pi}{4}\right)}{2^{2n}}.
\end{align}
Now let us suppose that $\Xi_{\mathbb{F}}(t)$ eventually takes one sign. Without the loss of generality, we assume $\Xi_{\mathbb{F}}(t)$ to be positive for $t \geq T$, for some large $T$. Then, we let
\begin{align*}
\lim_{\alpha \rightarrow \frac{\pi}{2}} \int_T^\infty \frac{\Xi_{\mathbb{F}}(t)}{\left(t^2 + \frac{1}{4}\right)}t^{2n}\cosh(\alpha t)\text{d}t = M.
\end{align*}
This $M$ will be positive since all the functions are positive in $[T,\infty)$. Thus for any $T' > T$,
\begin{align*}
\int_T^{T'} \frac{\Xi_{\mathbb{F}}(t)}{\left(t^2 + \frac{1}{4}\right)}t^{2n}\cosh(\alpha t)\text{d}t \leq M.
\end{align*}
Hence, letting $\alpha \rightarrow \frac{\pi}{2}$, we arrive
\begin{align*}
\int_T^{T'} \frac{\Xi_{\mathbb{F}}(t)}{\left(t^2 + \frac{1}{4}\right)}t^{2n}\cosh\left(\frac{\pi t}{2}\right)\text{d}t \leq M.
\end{align*}
Taking $T' \rightarrow \infty$ in the above inequality proves the fact that the integral 
$$
\int_0^\infty \frac{\Xi_{\mathbb{F}}(t)}{\left(t^2 + \frac{1}{4}\right)}t^{2n}\cosh\left(\alpha t\right)\text{d}t 
$$ 
is uniformly convergent with respect to $\alpha$. Thus, from \eqref{Integral with limit}, we deduce that
\begin{align*}
\int_0^\infty \frac{\Xi_{\mathbb{F}}(t)}{\left(t^2 + \frac{1}{4}\right)}t^{2n}\cosh\left(\frac{\pi t}{2}\right)\text{d}t = (-1)^{n+1}\frac{C_\mathbb{F}2^{r_1}\pi\cos\left(\frac{\pi}{4}\right)}{2^{2n}}, \quad \forall n.
\end{align*}
One can see that, from \eqref{Laurent series_at s=0_1st coeff}, $C_\mathbb{F}$ is negative for any number field $\mathbb{F}$, so for $n$ odd, it follows that
{\allowdisplaybreaks
\begin{align}\label{Bound for integral}
\int_0^\infty \frac{\Xi_{\mathbb{F}}(t)}{\left(t^2 + \frac{1}{4}\right)}t^{2n}\cosh\left(\frac{\pi t}{2}\right)\text{d}t &< 0 \nonumber \\
\implies \int_T^\infty \frac{\Xi_{\mathbb{F}}(t)}{\left(t^2 + \frac{1}{4}\right)}t^{2n}\cosh\left(\frac{\pi t}{2}\right)\text{d}t &< -\int_0^T \frac{\Xi_{\mathbb{F}}(t)}{\left(t^2 + \frac{1}{4}\right)}t^{2n}\cosh\left(\frac{\pi t}{2}\right)\text{d}t \nonumber \\
&< \ell \int_0^T t^{2n} \text{d}t \nonumber \\
&< \ell \, T^{2n+1}.
\end{align}}
In the penultimate step, we have used the fact that $-\frac{\Xi_{\mathbb{F}}(t)}{\left(t^2 + \frac{1}{4}\right)}\cosh\left(\frac{\pi t}{2}\right)$ is a continuous function in the closed interval $[0,T]$ and hence it can be bounded by a positive number $\ell$, where $\ell$ is independent of $n$. Since $\Xi_{\mathbb{F}}(t)$ is positive for $t \geq T$, so
\begin{align*}
\frac{\Xi_{\mathbb{F}}(t)}{\left(t^2 + \frac{1}{4}\right)}\cosh\left(\frac{\pi t}{2}\right) \geq c, \quad \text{for} \,\,\, \delta T \leq t \leq T(\delta + 1),
\end{align*}
where $\delta > 1$ and $c$ is some positive integer. Using the above lower bound in \eqref{Bound for integral}, we have
{\allowdisplaybreaks
\begin{align*}
\ell\,T^{2n+1} > \int_T^\infty \frac{\Xi_{\mathbb{F}}(t)}{\left(t^2 + \frac{1}{4}\right)}t^{2n}\cosh\left(\frac{\pi t}{2}\right)\text{d}t &\geq \int_{\delta T}^{T(\delta + 1)} \frac{\Xi_{\mathbb{F}}(t)}{\left(t^2 + \frac{1}{4}\right)}t^{2n}\cosh\left(\frac{\pi t}{2}\right)\text{d}t \\
&\geq c\int_{\delta T}^{T(\delta + 1)} t^{2n} \text{d}t \\
&\geq c\delta^{2n}T^{2n+1}.
\end{align*}}
Thus, $\ell > \delta^{2n} c$, which fails to hold for sufficiently large $n$, as $\delta > 1$. Hence our assumption is wrong. Therefore, $\Xi_{\mathbb{F}}(t)$ must change sign infinitely often, which implies that it has infinitely many real zeros and so has $\zeta_{\mathbb{F}}(s)$ on the critical line. This finishes the proof of Theorem \ref{BM2}.

\end{proof}

\section{{\bf Concluding Remarks}}
Using the Jacobi theta relation \eqref{Jacobi theta transformation}, Hardy proved that there are infinitely many non-trivial zeros of the Riemann zeta function on the critical line. Interestingly, the Jacobi theta relation is equivalent to the functional equation of $\zeta(s)$. In 1936, Ferrar proved Ramanujan-Koshliakov formula \eqref{Koshliakov Formula} using the functional equation for $\zeta^2(s)$. Moreover, he proved a general theta relation \eqref{Theta reln for zeta power k} assuming the functional equation for $\zeta^k(s)$, and we showed that the two are in fact equivalent, see Corollary \ref{Equivalence for zeta power k}. 
Motivated by this work, we have established an equivalence of the functional equation for $\zeta^{-k}(s)$ with $k \in \mathbb{N}$, see Corollary \ref{Theta reln for zeta power -k}. In particular, we deduced that the functional equation for $\zeta^{-1}(s)$ is equivalent to the well-known identity of Hardy, Littlewood and Ramanujan \eqref{Ramanujan identity for 1/zeta}.
It is straightforward to observe that the functional equations for $\zeta(s)$ and $\zeta^{-1}(s)$ are equivalent, which in turn implies that the corresponding theta identities, namely, Jacobi theta relation \eqref{Jacobi theta transformation} and Hardy, Littlewood and Ramanujan identity \eqref{Ramanujan identity for 1/zeta} are also equivalent. To the best of our knowledge, this observation has not been noted explicitly, in the existing literature. 

It is worthwhile to note that Ferrar \cite{Ferrar1936} was mainly interested to study solutions to the following theta relation 
\begin{align}\label{Ferrar eqn}
F(x) = F\left(\frac{1}{x}\right).
\end{align}
In this paper, we obtained number field analogue of the theta relation which is equivalent to the functional equation for $\zeta_\mathbb{F}^k(s)$ for any $k \in \mathbb{Z}^*$, see Theorems \ref{BM1}, \ref{BM3}. In fact, the theta relations \eqref{Theta reln for dedekind zeta power k} and \eqref{Theta reln for dedekind zeta power -k} arising from Theorems \ref{BM1}, \ref{BM3}, respectively, provide a new set of solutions to Ferrar's equation \eqref{Ferrar eqn}. 

Another important part of this paper involves showing the existence of infinitely many non-trivial zeros of $\zeta_\mathbb{F}(s)$ on the critical line using the theta relation \eqref{Theta reln for dedekind zeta}, see Theorem \ref{BM2}. 

Some other works include proving the existence of infinite number of non-trivial zeros of various other $L$-functions on the critical line. Zeros of $L$-functions of degree two in the Selberg class have been studied by Mukhopadhyay et. al. \cite{MSR2008}. Meher et. al. \cite{MPK2017, MPS2019} showed that there are infinitely many zeros on the critical line for the $L$-functions attached to modular forms of half-integral weight. Kim \cite{Kim} studied the infinitude of zeros of additively twisted $L$-function on the critical line. For further insights, one can see \cite{Conrey, DKMZ2018, R2025, RY24, Zhang}.

Establishing the existence of infinitely many non-trivial zeros on the critical line remains a challenging problem for higher degree $L$-functions. We speculate that every nice $L-$function will have an analogous theta type relation that might be useful to prove the existence of infinitely many non-trivial zeros on the critical line. Since the Dedekind zeta function serves as a key example, this approach could be useful for a broader class of $L$-functions. In an upcoming paper, we are working on the same problem for the Rankin-Selberg $L-$function and symmetric square $L-$function.

For the clarity of the reader, we present a table below which shows equivalence of several theta-type identities. The theta relations mentioned in each row are equivalent through their corresponding functional equations of the zeta functions listed in the middle column. This establishes equivalences between known classical identities, their dual forms, and further extends them to their analogues over number fields.
\vspace{-0.5cm}
\begin{table}[H]
\centering
\renewcommand{\arraystretch}{1.4}
\caption{Equivalence between theta-type identities and functional equations of $\zeta^k(s)$ and $\zeta_{\mathbb{F}}^k(s)$, $k \in \mathbb{Z}^*$.}
\vspace{0.2cm}
\begin{tabular}{|p{5cm}|c|p{5cm}|}
\hline
\textbf{Theta-type relations} & \textbf{Functional equation} & \textbf{Theta-type relations} \\
\hline
Jacobi theta relation \eqref{Jacobi theta transformation} & $\zeta(s) \Longleftrightarrow \zeta^{-1}(s)$ & Hardy, Littlewood and Ramanujan identity \eqref{Ramanujan identity for 1/zeta} \\
\hline
Ramanujan-Koshliakov identity \eqref{Koshliakov Formula} & $\zeta^{2}(s) \Longleftrightarrow \zeta^{-2}(s)$ & New theta identity \eqref{New id for zeta 2inv}  \\
\hline
Ferrar's identity \eqref{Theta reln for zeta power k} & $\zeta^{k}(s) \Longleftrightarrow \zeta^{-k}(s)$ & A generalization \eqref{New id for zeta inv} of Hardy, Littlewood and Ramanujan \\
\hline
Number field analogue \eqref{Theta reln for dedekind zeta} of Jacobi theta relation  & $\zeta_{\mathbb{F}}(s) \Longleftrightarrow \zeta_{\mathbb{F}}^{-1}(s)$ & An identity of Dixit, Gupta and Vatwani \eqref{Id of Dixit etal}\\
\hline
Number field analogue \eqref{Theta reln for dedekind zeta power k} of Ferrar's identity & $\zeta_{\mathbb{F}}^{k}(s) \Longleftrightarrow \zeta_{\mathbb{F}}^{-k}(s)$ & A generalization \eqref{Theta reln for dedekind zeta power -k} of Dixit, Gupta and Vatwani \\
\hline
\end{tabular}
\end{table}
Equivalence for the functional equation of $L-$functions have been studied by many mathematicians. Bochner \cite{Bochner} studied a class of functions that satisfy the functional equation with single gamma factor and showed that it is equivalent to a modular-type relation which is nothing but theta-type relation. Chandrasekharan and Narasimhan \cite{CN1961, CN1962} also studied different equivalent identities for the functional equation of $L-$functions with multiple gamma factors. For further information, one can refer to a recent work of Roy, Sahoo and Vatwani \cite{RSV24}.

{\bf Acknowledgement}
We express our sincere thanks to Prof. Bruce Berndt and Prof. Atul Dixit for reading the manuscript and suggesting the related papers on this literature. The first author acknowledges support from the Prime Minister Research Fellowship (PMRF), Government of India, under Grant No. 2102227. The last author gratefully acknowledges funding from the Science and Engineering Research Board (SERB), India, through the MATRICS grant (File No. MTR/2022/000545) and the CRG grant (File No. CRG/2023/002122). We would also like to thank IIT Indore for fostering a supportive research environment.

{\bf Conflict of Interest}
The authors declare that there are no conflicts of interest.

{\bf Data Availability Statement}
The authors confirm that this manuscript does not involve any associated data.


\begin{thebibliography}{99}
\bibitem{AGM22}
A. Agarwal, M. Garg, B. Maji, Riesz-type criteria for the Riemann hypothesis, Proc. Amer. Math. Soc. {\bf 150} (2022), 5151--5163.


\bibitem{GB}
G.~E.~Andrews and B.~C.~Berndt, \emph{Ramanujan's Lost Notebook}, Part IV, Springer, New York, 2013.


\bibitem{ACK2015}
T. Arai, K. Chakraborty, and S. Kanemitsu, \emph{On modular relations}, In Number theory, volume 11 of Ser. Number Theory Appl., pages 1--64. World Sci. Publ., Hackensack, NJ, 2015.

\bibitem{BK23}
S.~Banerjee and R.~Kumar, \emph{Equivalent criterion for the grand Riemann Hypothesis associated to Maass cusp forms}, Proc. Roy. Soc. Edinburgh Sect. A, {\bf 154} (2024), 1348--1363.

\bibitem{Berndt1}
B.~C.~Berndt, \emph{On the zeros of a class of Dirichlet series I}, Ill. J. Math., {\bf 14} (1970), 244--258.


\bibitem{Berndt2}
B.~C.~Berndt, \emph{On the zeros of a class of Dirichlet series II}, Ill. J. Math., {\bf 14} (1970), 678--691.

\bibitem{Berndt3}
B.~C.~Berndt, \emph{The number of zeros of the Dedekind zeta-function on the critical line}, J. Number Theory, {\bf 3} (1971), 1--6.

\bibitem{Berndt4}
B.~C.~Berndt, Ramanujan’s Notebooks, Part V, Springer-Verlag, New York, 1998.



\bibitem{Riemann}
R.~Bernhard, \emph{Ueber die Anzahl der Primzahlen unter einer gegebenen Gr\"osse}, Monatsberichte der Berliner Akademie, 145--155, 1859.

\bibitem{Bochner}
S.~Bochner, \emph{Some properties of modular relations}, Ann. of Math., {\bf 53} (1951), 332–363.

\bibitem{CN1961} 
K.~Chandrasekharan and R.~Narasimhan,  \emph {Hecke's functional equation and arithmetical identities}, Ann. of Math., {\bf 74} (1961), 1--23.


\bibitem{CN1962} 
K.~Chandrasekharan and R.~Narasimhan,  \emph {Functional equations with multiple gamma factors and the average order of arithmetical functions}, Ann. of Math., {\bf 76} (1962), 93--136.


\bibitem{CN1963} 
K.~Chandrasekharan and R.~Narasimhan,  \emph {The approximate functional equation for a class of zeta-functions}, Math. Annalen, {\bf 152} (1963), 30--64.


\bibitem{CN1968}
K.~Chandrasekharan and R.~Narasimhan, \emph{Zeta-functions of ideal classes in quadratic fields and their zeros on the critical line}, Comment. Math. Helv., {\bf 43} (1968), 18--30.



\bibitem{Conrey} J.~B.~Conrey, \emph{More than two fifths of the zeros of the Riemann zeta function are on the critical line}, J. Reine Angew. Math. {\bf 399} (1989), 1--26.


\bibitem{Dixit12}
A.~Dixit, \emph{Character analogues of Ramanujan-type integrals involving the Riemann $\Xi$-function}, Pacific J. Math., {\bf 255} (2012), 317--348.

\bibitem{Dixit15}
A. Dixit, A. Roy, and A. Zaharescu, \emph{Ramanujan-Hardy-Littlewood-Riesz phenomena for Hecke forms}, J. Math. Anal. Appl., {\bf 426} (2015), 594--611.

\bibitem{Dixit16}
A. Dixit, A. Roy, and A. Zaharescu, \emph{Riesz-type criteria and theta transformation analogues}, J. Number Theory, {\bf 160} (2016), 385--408.

\bibitem{DGV2022}
A.~Dixit, S.~Gupta, and A.~Vatwani, \emph{A modular relation involving non-trivial zeros of the Dedekind zeta function, and the generalized Riemann hypothesis}, J. Math. Anal. Appl., {\bf 515} (2022), 126435.


\bibitem{DKMZ2018}
A.~Dixit, R.~Kumar, B.~Maji, and A.~Zaharescu, \emph{Zeros of combinations of the Riemann $\Xi$-function and the confluent hypergeometric function on bounded vertical shifts}, J. Math. Anal. Appl., {\bf 466} (2018), 307--323.


\bibitem{Ferrar1936}
W.~L.~Ferrar, \emph{Some solutions of the equation $F(t) = F(t^{-1})$}, J. Lond. Math. Soc., {\bf 11} (1936), 99--103.

\bibitem{GM23}
M. Garg and B. Maji, \emph{Hardy-Littlewood-Riesz type equivalent criteria for the generalized Riemann hypothesis}, Monatsh. Math., {\bf 201} (2023), 771--788.


\bibitem{Hardy}
G.~H.~Hardy, \emph{Sur les z\'eros de la fonction $\zeta(s)$ de Riemann}, C. R. Acad. Sci. Paris, {\bf 158} (1914), 1012--1014.

\bibitem{HLR1916}
G.H.~Hardy and J.E.~Littlewood, \emph{Contributions to the theory of the Riemann zeta-function and the theory of the distribution of primes}, Acta Math. {\bf 41} (1916), 119--196.

\bibitem{Brown77}
D.R. Heath-Brown, \emph{On the density of the zeros of the Dedekind zeta-function}, Acta Arith., {\bf 33} (1977), 169--181.

\bibitem{Hecke1937}
E.~Hecke, \emph{\"Uber Dirichlet‑Reihen mit Funktionalgleichung und ihre Nullstellen auf der Mittelgeraden}, Bayer. Akad. Wiss. Math.-Natur. Kl. S.-B. II, {\bf 8} (1937), 73--95. 

\bibitem{KT} 
S. ~Kanemitsu and H. ~Tsukada,  \emph{Contributions to the theory of zeta-functions},  The Modular Relation Supremacy, Series on Number Theory and Its Applications, vol.10, World Sci., Singapore, 2015.

\bibitem{Kim}
D.~Kim, \emph{Infinitely many zeros of additively twisted L-functions on the critical line}, J. Number Theory, {\bf 253} (2023), 157--187.


\bibitem{Koshliakov}
N.~S.~Koshliakov, \emph{On Voronoi’s sum-formula}, Mess. Math, {\bf 58} (1929), 30--32.

\bibitem{KoshI}
N.~S.~Koshliakov, \emph{Investigation of some questions of analytic theory of
the rational and quadratic fields}, I (Russian), Izv. Akad. Nauk SSSR, Ser.
Mat., {\bf 18} (1954), 113--144.


\bibitem{Landau1903}
E.~Landau, \emph{\"Uber die Zahlentheoretische Funktion $\mu(k)$}, Wiener Sitzungsber, Edmund Landau Collected Works, Thales Verlag, {\bf 112} (1903), 60--93.

\bibitem{MPK2017}
J. Meher, S. Pujahari, S. Kotyada, \emph{Zeros of L‐functions attached to modular forms of half‐integral weight}, Bull. London Math. Soc., {\bf 49} (2017), 926--936.


\bibitem{MPS2019}
J. Meher, S. Pujahari, K.~D.~Shankhadhar, \emph{Zeros of L-functions attached to cusp forms of half-integral weight}, Proc. Amer. Math. Soc., {\bf 147} (2019), 131--143.


\bibitem{MSR2008}
A.~Mukhopadhyay, K.~Srinivas, and K.~Rajkumar, \emph{On the zeros of functions in the Selberg class}, Funct. Approx. Comment. Math., {\bf 38} (2008), 121--130.

\bibitem{OLBC2010}
F.W.J.~Olver, D.W.~Lozier, R.F.~Boisvert, C.W.~Clark (Eds.), NIST Handbook of Mathematical Functions, Cambridge University Press, Cambridge, 2010.

\bibitem{Ramanujan}
S.~Ramanujan, \emph{The Lost Notebook and Other Unpublished Papers}, Narosa, New Delhi, 1988.


\bibitem{R2025}
P.~Ribeiro, \emph{On the number of zeros of L-functions attached to cusp forms of half-integral weight}, J. Number Theory, https://doi.org/10.1016/j.jnt.2025.02.014

\bibitem{RY24}
P.~Ribeiro and S.~Yakubovich, \emph{On the Epstein zeta function and the zeros of a class of Dirichlet series}, J. Math. Anal. Appl., {\bf 530} (2024), 127590.


\bibitem{RSV24}
A.~Roy, J.~Sahoo, A.~Vatwani, \emph{Equivalence between the functional equation and Voronoi-type summation identities for a class of functions}, Proc. Roy. Soc. Edinburgh Sect. A, (2024), 1--41, DOI:10.1017/prm.2024.107

\bibitem{Steen}
S.~W.~P.~Steen, \emph{Divisor functions: their differential equations and recurrence formulas}, Proc. Lond. Math. Soc. {\bf 2} (1930), 47--80.

\bibitem{Titchmarsh}
E.C. Titchmarsh, \emph{The Theory of the Riemann Zeta-Function}, 2nd ed., Clarendon Press, Oxford University Press, New York, 1986.

\bibitem{Zhang}
W.~Zhang, \emph{A note on simple zeros related to Dedekind zeta functions}, Ramanujan J., {\bf 64} (2024), 333--341.




\end{thebibliography}
\end{document}